\documentclass[3p]{article}
\usepackage{amsfonts,amsmath,amssymb,amsthm,stmaryrd,hyperref}
\usepackage[francais,english]{babel}
\usepackage[T1]{fontenc}
\usepackage[ansinew]{inputenc}  
\usepackage{graphics}
\usepackage{graphicx}
\usepackage{tikz}
\usepackage{multicol}
\usetikzlibrary{patterns}
\usepackage{amssymb,subfigure,lineno,url}
\usepackage{float}

\newtheorem{theorem}{Theorem}
\newtheorem{definition}{Definition}
\newtheorem{lemma}{Lemma}
\newtheorem{proposition}{Proposition}
\newtheorem{rmk}{ Remark}

\begin{document}
\title{On nonexistence of solutions to some time-space fractional evolution equations with transformed space argument}

\author{Mokhtar Kirane\footnote{\noindent 
Department of Mathematics, Faculty of Arts and Science, Khalifa University, P.O. Box: 127788,  Abu Dhabi, UAE. King Abdulaziz university, NAAM Group, Department of Mathematics, Faculty of Science, Jeddah 21589, Saudi Arabia; mokhtar.kirane@ku.ac.ae.
 \newline \indent $\,\,{}^{a}$  Department of Mathematics, Sultan Qaboos University,  FracDiff Research Group (DR/RG/03),  P.O. Box 46, Al-Khoud 123, Muscat, Oman; ahmad.fino01@gmail.com; a.fino@squ.edu.om.
  \newline \indent $\,\,{}^{b}$  King Abdulaziz university, NAAM Group, Department of Mathematics, Faculty of Science, Jeddah 21589, Saudi Arabia; bashirahmad qau@yahoo.com.}, Ahmad Z. Fino$^{\,a}$, Bashir Ahmad$^{\,b}$}

\date{}
\maketitle

\begin{abstract}
Some results on nonexistence of nontrivial solutions to some time and space fractional differential evolution equations with transformed space argument are obtained via the nonlinear capacity method. The analysis is then used for a $2\times2$ system of equations with transformed space arguments.
\end{abstract}

\maketitle



\section{Introduction}
There is a sizeable number of works on nonexistence of solutions or existence of blowing-up solutions to different type of classical steady or evolution equations and recently for fractional differential equations \cite{Hu, Quittner, Mitidieri, Guedda, Kirane3, Kirane2}. In \cite{Alsaedi2}, Ahmad, Alsaedi and Kirane collected all results concerning blowing-up or growing-up solutions for delay differential equations. Differential equations with transformed arguments have been treated with respect to many points (existence, large time behavior, etc) in many works \cite{Antonevich, Aftabizadeh, Andreev, Borok, Burlutskayaa, Blasio, Gupta1, Gupta2, Kirane1, Vorontsov, Wiener}, but principally in the works of Przeworska-Rolewicz \cite{Przeworska} and Skubachevskii \cite{Skubachevskii}. Very recently, Salieva in \cite{Salieva} obtained results on nonexistence of solutions for nonlinear differential inequalities with transformed argument. Salieva obtained sufficient conditions for the nonexistence of solutions to some classes of differential inequalities and
systems of inequalities. Here, we extend the study of Salieva \cite{Salieva} to time-space fractional differential equations.\\

 We consider first the time-space fractional evolution equation with space transformed argument
\begin{equation}\label{1}
\left\{
\begin{array}{ll}
\,\,\displaystyle{{\bf D}^\alpha_{0|t}u(t,x)+(-\Delta)^{\delta/2} u(t,x)=|u(t,g(x))|^p,}&\displaystyle {t>0,\,\,x\in\mathbb{R}^{d},}\\
\\
\displaystyle{u(x,0)=u_0(x)},&\displaystyle{x\in {\mathbb{R}^d}, }
\end{array}
\right. 
\end{equation}
where $0<\alpha\leq1$, $0<\delta\leq2$, $p>1$, and $d\geq1$. ${\bf D}^\alpha_{0|t}$ stands for the Caputo fractional derivative for $0<\alpha<1$ and for the standard partial derivative in time $\partial_t$ when $\alpha=1$, $(-\Delta)^{\delta/2}$, $0<\delta<2$, is the fractional power of the laplacian that will be defined here below, and $g\in C^1(\mathbb{R}^d,\mathbb{R}^d)$ is an invertible mapping satisfying:
\begin{itemize}
\item[(A1)] there exists a constant $c_0>0$ such that $|J_g^{-1}(x)|\geq c_0>0$ for all $x\in \mathbb{R}^d$ ($J$ is the jacobian matrix);
\item[(A2)] $|g(x)|\geq |x|$ for all $x\in \mathbb{R}^d$.
\end{itemize}
As examples of such $g$ we mention (cf \cite{Salieva}):
\begin{enumerate}
\item The dilation mapping $g(x)=kx$, for any $k\in\mathbb{R}$ with $|k|>1$, that
satisfies (A1) with $c_0=|k|^{-d}$ and (A2).
\item  The rotation transform $g(x) =Ax$, where $A$ is a $d\times d$ unitary matrix
(so $|g(x)| =|x|$ for all $x\in \mathbb{R}^d$) that satisfies (A1) with $c_0=1$ and (A2).\\

In some situations (A2) can be replaced by a weaker one:\\
\begin{itemize}
\item[(A$2^*$)] there exists positive constants $c_0$ and $\rho$ such that $|g(x)|\geq c_0|x|$ for all $x\in \mathbb{R}^d\setminus B_\rho(0)$ ($c_0$ may be taken $c_0\leq1$).
\end{itemize}
\item The shift transform $g(x)=x-x_0$ for a fixed $x_0\in \mathbb{R}^d$ with $c=1$, $c_0=1/2$ and $\rho=2|x_0|$.
\end{enumerate}

Then, we consider the equation
\begin{equation}\label{14}
\left\{
\begin{array}{ll}
\,\,\displaystyle{{\bf D}^\beta_{0|t}u(t,x)+(-\Delta)^{\delta/2} u(t,x)+{\bf D}^\alpha_{0|t}u(t,x)=|u(t,g(x))|^p,}&\displaystyle {t>0,\,\,x\in\mathbb{R}^{d},}\\
\\
\displaystyle{u(x,0)=u_0(x),\,\,u_t(x,0)=u_1(x)},&\displaystyle{x\in {\mathbb{R}^d}, }
\end{array}
\right.
 \end{equation}
where $1<\beta\leq2$, $0<\delta\leq2$, $0<\alpha\leq1$, $p>1$, and $d\geq1$. Finally, we consider the following $2\times 2$ system 
\begin{equation}\label{15}
\left\{
\begin{array}{ll}
\,\,\displaystyle{{\bf D}^\gamma_{0|t}u(t,x)+(-\Delta)^{\mu/2} u(t,x)=|v(t,g(x))|^p,}&\displaystyle {t>0,\,\,x\in\mathbb{R}^{d},}\\\\
\,\,\displaystyle{{\bf D}^\theta_{0|t}u(t,x)+(-\Delta)^{\sigma/2} u(t,x)=|u(t,f(x))|^q,}&\displaystyle {t>0,\,\,x\in\mathbb{R}^{d},}\\\\
\displaystyle{u(x,0)=u_0(x),\,\,v(x,0)=v_0(x)},&\displaystyle{x\in {\mathbb{R}^d},}
\end{array}
\right. 
\end{equation}
where $0<\gamma,\theta\leq1$, $0<\sigma,\mu\leq2$, $p,q>1$, and $d\geq1$.\\

\textbf{Notations}
\begin{itemize}
\item Constants $C$ and $C_i$ with $i \in \mathbb{N}$ stand for suitable positive constants.
\item For given nonnegative $f$ and $g$, we write $f\lesssim g$ if $f\le Cg$, for constant $C>0$. 
\end{itemize}


\section{Main results}\label{sec2}

For a weight function $w$ and $1\leq r<\infty$, let $L^r_w$ denote the space of all real-valued measurable functions $f$ such that $f|w|^{1/r}\in L^r$, the usual Lebesgue space. Let
$$X_{\delta,T}=\{\varphi\in C([0,\infty),H^\delta(\mathbb{R}^d))\cap C^1([0,\infty),L^2(\mathbb{R}^d)), \hbox{such that supp$\varphi\subset Q_T$}\},$$
$$Y_{\delta,T}=\{\varphi\in C([0,\infty),H^\delta(\mathbb{R}^d))\cap C^2([0,\infty),L^2(\mathbb{R}^d)), \hbox{such that supp$\varphi\subset Q_T$}\},$$
where $Q_T:=[0,T]\times\mathbb{R}^d$, and the fractional Sobolev space $H^\delta(\mathbb{R}^d)$, $\delta\in(0,2)$, is defined by
$$H^\delta(\mathbb{R}^d)=\{u\in L^2(\mathbb{R}^d); (-\Delta)^{\delta/2}u\in L^2(\mathbb{R}^d)\},$$
endowed with the norm
$$\|u\|_{H^\delta(\mathbb{R}^d)}=\|u\|_{L^2(\mathbb{R}^d)}+\left\|(-\Delta)^{\delta/2}u\right\|_{L^2(\mathbb{R}^d)}.$$

\begin{definition}\label{definitionweak}
Let $u_0\in L^2(\mathbb{R}^d)$ and $T>0.$ A function 
$$u\in L^1((0,T),L^{2}(\mathbb{R}^d))\cap L^p((0,T),L_{|J_g^{-1}|}^{2p}(\mathbb{R}^d)),$$
is said to be a weak solution of \eqref{1} on $[0,T)\times\mathbb{R}^d$ if 
\begin{eqnarray*}
&{}&\int_{Q_T}|u(t,g(x))|^p\varphi(t,x)\,dt\,dx+\int_{Q_T}u_0(x)\,D^\alpha_{t|T}\varphi(t,x)\,dt\,dx\nonumber\\
&{}&=\int_{Q_T}u(t,x)\,D^\alpha_{t|T}\varphi(t,x)\,dt\,dx+\int_{Q_T}u(t,x)(-\Delta)^{\delta/2}\varphi(t,x)\,dt\,dx,
\end{eqnarray*}
holds for all $\varphi\in X_{\delta,T}$. We denote the lifespan for the weak solution by
$$T_w(u_0):=\sup\{T\in(0,\infty];\,\,\hbox{there exists a unique weak solution $u$ of \eqref{1}}\}.$$
Moreover, if $T>0$ can be arbitrary chosen, i.e. $T_w(u_0)=\infty$, then $u$ is called a global weak solution of \eqref{1}.
\end{definition}

\begin{theorem}\label{theo1}${}$
Let $u_0\in L^1(\mathbb{R}^d)\cap L^2(\mathbb{R}^d)$, $0<\alpha\leq1$, $0<\delta\leq2$, $p>1$, and $d\geq1$. Assume that $g$ satisfies conditions $(\textup{A}1)$-$(\textup{A}2)$. If
\begin{equation}\label{C1}
\left\{\begin{array}{ll}
p< p_*&\quad\hbox{when}\,\,\alpha\in(0,1),\\
{}\\
p\leq p_*&\quad\hbox{when}\,\,\alpha=1,\\
\end{array}
\right.
\end{equation}
with
$$p_*=1+\frac{\alpha\delta}{\alpha d+\delta(1-\alpha)},$$
then problem \eqref{1}  admits no global nontrivial weak solutions.
\end{theorem}

Next, define the weak solution of the corresponding $2\times2$ system.

\begin{definition}\label{definitionweak2}
Let $u_0,v_0\in L^2(\mathbb{R}^d)$, and $T>0.$ A couple of function $(u,v)$ such that
$$u,v\in L^1((0,T),L^{2}(\mathbb{R}^d)),\quad u\in L^q((0,T),L_{|J_f^{-1}|}^{2q}(\mathbb{R}^d)),\quad v\in L^p((0,T),L_{|J_g^{-1}|}^{2p}(\mathbb{R}^d)),$$
is said to be a weak solution of \eqref{15} on $[0,T)\times\mathbb{R}^d$ if 
\begin{eqnarray*}
&{}&\int_{Q_T}|v(t,g(x))|^p\varphi(t,x)\,dt\,dx+\int_{Q_T}u_0(x)\,D^\gamma_{t|T}\varphi(t,x)\,dt\,dx\nonumber\\
&{}&=\int_{Q_T}u(t,x)\,D^\gamma_{t|T}\varphi(t,x)\,dt\,dx+\int_{Q_T}u(t,x)(-\Delta)^{\mu/2}\varphi(t,x)\,dt\,dx,
\end{eqnarray*}
and
\begin{eqnarray*}
&{}&\int_{Q_T}|u(t,f(x))|^q\psi(t,x)\,dt\,dx+\int_{Q_T}v_0(x)\,D^\theta_{t|T}\psi(t,x)\,dt\,dx\nonumber\\
&{}&=\int_{Q_T}v(t,x)\,D^\theta_{t|T}\psi(t,x)\,dt\,dx+\int_{Q_T}v(t,x)(-\Delta)^{\sigma/2}\psi(t,x)\,dt\,dx,
\end{eqnarray*}
hold for all $\varphi\in X_{\mu,T}$, $\psi\in X_{\sigma,T}$. We denote the lifespan for the weak solution by
$$T_w(u_0,v_0):=\sup\{T\in(0,\infty];\,\,\hbox{there exists a unique weak solution $(u,v)$ of \eqref{14}}\}.$$
Moreover, if $T>0$ can be arbitrary chosen, i.e. $T_w(u_0,v_0)=\infty$, then $u$ is called a global weak solution of \eqref{15}.

\end{definition}

\begin{theorem}\label{theo2}${}$
Let $u_0,v_0\in L^1(\mathbb{R}^d)\cap L^2(\mathbb{R}^d)$, $0<\gamma,\theta\leq1$, $0<\mu,\sigma\leq2$, $p,q>1$, and $d\geq1$; assume that $\theta p+\gamma pq-pq+1>0$, $\gamma q+\theta pq-pq+1>0$, and that $f,g$ satisfy conditions $(\textup{A}1)$-$(\textup{A}2)$. If
$$
 d<\max\left\{\overline{D}\,;\,\overline{E}\right\},
$$
or
$$
\left\{\begin{array}{ll}
\displaystyle d=\overline{D},&\hbox{when}\,\,\,\theta=1,\,\,\gamma\neq1,\,\hbox{and}\,\,pq>1/(1-\gamma),\\\\
\displaystyle d=\overline{D},&\hbox{when}\,\,\,\theta=1,\,\,\gamma\leq1,\,\hbox{and}\,\,\gamma\leq\frac{p(q-1)}{q(p-1)},\\\\
\displaystyle d=\overline{E},&\hbox{when}\,\,\,\gamma=1,\,\,\theta\neq1,\,\hbox{and}\,\,pq>1/(1-\theta),\\\\
\displaystyle d=\overline{E},&\hbox{when}\,\,\,\gamma=1,\,\,\theta\leq1,\,\hbox{and}\,\,\theta\leq\frac{q(p-1)}{p(q-1)},
\end{array}
\right.
$$
where 
$$\overline{D}:=\min\{\max\{D_1,D_2\}\,;\,\max\{D_3,D_4\}\}\quad\hbox{and}\quad\overline{E}:=\min\{\max\{E_1,E_2\}\,;\,\max\{E_3,E_4\}\},$$
 with
$$D_1=\frac{\theta\sigma p+\theta\mu pq-\sigma pq+\sigma}{\theta(pq-1)},\qquad D_2=\frac{\mu(\theta p+\gamma pq- pq+1)}{\gamma(pq-1)},$$
$$ D_3=\frac{\gamma\sigma p+\gamma\mu pq-\mu pq+\mu}{\gamma(pq-1)},\qquad D_4=\frac{\sigma(\theta p+\gamma pq- pq+1)}{\theta(pq-1)},$$
$$E_1=\frac{\theta\mu q+\theta\sigma pq-\sigma pq+\sigma}{\theta(pq-1)},\qquad E_2=\frac{\mu(\gamma q+\theta pq- pq+1)}{\gamma(pq-1)},$$
$$E_3=\frac{\gamma\mu q+\gamma\sigma pq-\mu pq+\mu}{\gamma(pq-1)},\qquad E_4=\frac{\sigma(\gamma q+\theta pq- pq+1)}{\theta(pq-1)},$$
then system \eqref{15}  admits no global nontrivial weak solutions.
\end{theorem}

\begin{rmk}${}$
\begin{enumerate}
\item If $\gamma=\theta=1$ and $\mu=\sigma=2$, then $\overline{D}=2(p+1)/(pq-1)$ and $\overline{E}=2(q+1)/(pq-1)$, which are the same exponent found in \cite{Herrero}.
\item If $\gamma=\theta=1$, then 
$$\overline{D}:=\min\{\max\{D_1,D_2\}\,;\,D_3\}\quad\hbox{and}\quad\overline{E}:=\min\{E_1\,;\,\max\{E_3,E_4\}\}.$$
If $\mu\leq \sigma$, then 
$$\overline{D}:=\max\{D_1,D_2\}\geq D_1=\frac{\sigma p+\mu pq-\sigma pq+\sigma}{pq-1},\quad\hbox{and}\quad\overline{E}:=E_1=\frac{\mu q+\sigma}{pq-1}.$$
If $\mu\geq \sigma$, then 
$$\overline{D}:=D_3=\frac{\sigma p+\mu}{pq-1},\quad\hbox{and}\quad\overline{E}:=\max\{E_3,E_4\}\geq E_3=\frac{\mu q+\sigma pq-\mu pq+\mu}{pq-1}.$$
As a conclusion, our result is an improvement of \cite{Guedda} and without any additional conditions.
\end{enumerate}
\end{rmk}

\begin{definition}\label{definitionweak3}
Let $u_0,u_1\in L^2(\mathbb{R}^d)$ and $T>0.$ A function 
$$u\in L^1((0,T),L^{2}(\mathbb{R}^d))\cap L^p((0,T),L_{|J_g^{-1}|}^{2p}(\mathbb{R}^d)),$$
is said to be a weak solution of \eqref{14} on $[0,T)\times\mathbb{R}^d$ if 
\begin{eqnarray*}
&{}&\int_{Q_T}|u(t,g(x))|^p\varphi(t,x)\,dt\,dx+\int_{Q_T}u_0(x)\,\left[D^\alpha_{t|T}\varphi(t,x)+D^{\beta}_{t|T}\varphi(t,x)\right]\,dt\,dx+\int_{Q_T}u_1(x)\,D^{\beta-1}_{t|T}\varphi(t,x)\,dt\,dx\nonumber\\
&{}&=\int_{Q_T}u(t,x)\,D^{\beta}_{t|T}\varphi(t,x)\,dt\,dx+\int_{Q_T}u(t,x)\,D^\alpha_{t|T}\varphi(t,x)\,dt\,dx+\int_{Q_T}u(t,x)(-\Delta)^{\delta/2}\varphi(t,x)\,dt\,dx,
\end{eqnarray*}
holds for all $\varphi\in Y_{\delta,T}$. We denote the lifespan for the weak solution by
$$T_w(u_0,u_1):=\sup\{T\in(0,\infty];\,\,\hbox{there exists a unique weak solution $u$ of \eqref{14}}\}.$$
Moreover, if $T>0$ can be arbitrary chosen, i.e. $T_w(u_0,u_1)=\infty$, then $u$ is called a global weak solution of \eqref{14}.
\end{definition}

\begin{theorem}\label{theo3}${}$
Let $u_0,u_1\in L^1(\mathbb{R}^d)\cap L^2(\mathbb{R}^d)$, $1<\beta\leq2$, $0<\delta\leq2$, $0<\alpha\leq1$, $p>1$, and $d\geq1$. Assume that $g$ satisfies conditions $(\textup{A}1)$-$(\textup{A}2)$. If
\begin{equation}\label{C3}
\left\{\begin{array}{ll}
p< p_*&\quad\hbox{when}\,\,\alpha\in(0,1),\\
{}\\
p\leq p_*&\quad\hbox{when}\,\,\alpha=1,\\
\end{array}
\right.
\end{equation}
then problem \eqref{14}  admits no global nontrivial weak solutions.
\end{theorem}


\section{Preliminaries}\label{sec3}
This section is devoted to collect some preliminaries needed in our proofs.
\begin{definition} (Absolutely continuous functions)\\
A function $f:[a,b]\rightarrow\mathbb{R}$, with $a,b\in\mathbb{R}$, is absolutely continuous if and only if there exists a Lebesgue summable function $\varphi\in L^1(a,b)$ such that 
$$f(t)=f(a)+\int_{a}^t\varphi(s)\,ds.$$ 
The space of such functions is denoted by $AC[a,b]$. Moreover, for all $m\geq0$, we define
$$AC^{m+1}[a,b]:= \big\{f:[a,b]\rightarrow\mathbb{R}\;\text{such that}\;D^m f\in
AC[a,b]\big\},$$
where $D^m= \displaystyle\frac{d^m}{dt^m}$ is the usual $m$ times derivative.
\end{definition}

\begin{definition}(Riemann-Liouville fractional derivatives) \cite[Chapter~1]{Samko} \\
Let $f\in AC[0,T]$ with $T>0$. The Riemann-Liouville left- and right-sided fractional derivatives of order $\alpha$ are defined by
$$
D^\alpha_{0|t}f(t):=\frac{1}{\Gamma(1-\alpha)}\frac{d}{dt}\int_{0}^t(t-s)^{-\alpha}f(s)\,ds, \quad t>0,\quad\alpha\in(0,1),
$$
$$
D^\alpha_{0|t}f(t):=\frac{1}{\Gamma(2-\alpha)}\frac{d^2}{dt^2}\int_{0}^t(t-s)^{-(\alpha-1)}f(s)\,ds, \quad t>0,\quad\alpha\in(1,2),
$$
$$
D^\alpha_{t|T}f(t):=-\frac{1}{\Gamma(1-\alpha)}\frac{d}{dt}\int_t^{T}(s-t)^{-\alpha}f(s)\,ds, \quad t<T,\quad\alpha\in(0,1),
$$
and
$$
D^\alpha_{t|T}f(t):=\frac{1}{\Gamma(2-\alpha)}\frac{d^2}{dt^2}\int_t^{T}(s-t)^{-(\alpha-1)}f(s)\,ds, \quad t<T,\quad\alpha\in(1,2),
$$
where $\Gamma$ is the Euler gamma function.
\end{definition}

\begin{definition}[Caputo fractional derivatives] \cite[Chapter~1]{Samko}\\
Let $f\in AC[0,T]$ with $T>0$. The Caputo left- and right-sided fractional derivatives of order $\alpha$ exists almost everywhere on $[0,T]$ and are defined by
$$
{}^cD^\alpha_{0|t}f(t):=D^\alpha_{0|t}[f(t)-f(0)], \quad\forall\, t>0,\quad\alpha\in(0,1),
$$
and
$$
{}^cD^\alpha_{0|t}f(t):=D^\alpha_{0|t}[f(t)-f(0)-tf^\prime(0)], \quad\forall\, t>0,\quad \alpha\in(1,2).
$$
\end{definition}

Given $T>0$, let $\phi: [0,\infty) \to \mathbb{R}$ be
\begin{equation}\label{w}
\displaystyle \phi(t)=\left(1-\frac{t}{T}\right)_{+}^{\mu}=\left\{\begin{array}{ll}
 \left(1-\frac{t}{T}\right)^{\mu}&\quad\hbox{if}\,\,\,0\leq t\leq T,\\\\
 0&\quad\hbox{if}\,\,\,t\geq T,\\
 \end{array}
 \right.
\end{equation}
where $\mu \gg1$ is big enough.

\noindent{\bf Remark:} Here one should recognize that the choice of such a function $\phi$ to prove a blow-up result has been for the first time used by Furati and Kirane in \cite{FuratiKirane}.\\

Later on, we need the following properties concerning the function $\phi$.

\begin{lemma}\cite[(2.45), p. 40]{Samko}\label{lemma1} \\
Let $T>0$, $0<\alpha<1$ and $m\geq0$. For all $t\in[0,T]$, we have
\begin{equation}\label{6}
D_{t|T}^{m+\alpha}\phi(t)=\frac{\Gamma(\mu+1)}{\Gamma(\mu+1-m-\alpha)}T^{-(m+\alpha)}(1-t/T)^{\mu-\alpha-m}.
\end{equation}
\end{lemma}

\begin{lemma}\label{lemma2}
Let $T>0$, $0<\alpha<1$, $m\geq0$ and $p>1$. Then, we have
\begin{equation}
\int_0^T \big(\phi(t)\big)^{-\frac{1}{p-1}}|D_{t|T}^{m+\alpha}\phi(t)|^{\frac{p}{p-1}}\,dt=C\,T^{1-(m+\alpha)\frac{p}{p-1}},
\end{equation}
and
\begin{equation}\label{8}
\int_0^TD_{t|T}^{m+\alpha}\phi(t)\,dt= C\,T^{-(m+\alpha)}.
\end{equation}
\end{lemma}

\begin{proof}
Using Lemma \ref{lemma1} we have
\begin{align*}
\int_{0}^T \big(\phi(t)\big)^{-\frac{1}{p-1}}|D_{t|T}^{m+\alpha}\phi(t)|^{\frac{p}{p-1}}\,dt &= C\,T^{-(m+\alpha)\frac{p}{p-1}}\int_{0}^T \big(\phi(t)\big)^{-\frac{1}{p-1}}\big(\phi(t)\big)^{\frac{p({\mu}-\alpha-m)}{(p-1)\mu}}\,dt\\
&= C\,T^{-(m+\alpha)\frac{p}{p-1}}\int_{0}^T(1-t/T)^{\mu-(m+\alpha)\frac{ p}{p-1}}\,dt\\
&= C\,T^{1-(m+\alpha)\frac{p}{p-1}}\int_{0}^1(1-s)^{\mu-(m+\alpha)\frac{ p}{p-1}}\,ds\\
&= C\,T^{1-(m+\alpha)\frac{p}{p-1}},
\end{align*}
where we have used $\mu \gg1$ to guarantee the integrability of the last integral. We obtain \eqref{8} in the same way.
\end{proof}
${}$
\begin{definition}\cite{Silvestre}\label{def1}
Let $u\in\mathcal{S}$  be the Schwartz space of rapidly decaying $C^\infty$ functions in $\mathbb{R}^d$ and $s \in (0,1)$. The fractional Laplacian $(-\Delta)^s$ in $\mathbb{R}^d$ is a non-local operator given by
\begin{eqnarray*}
 (-\Delta)^s v(x)&:=& C_{d,s}\,\, p.v.\int_{\mathbb{R}^d}\frac{v(x)- v(y)}{|x-y|^{d+2s}}\,dy\\\\
&=&\left\{\begin{array}{ll}
\displaystyle C_{d,s}\,\int_{\mathbb{R}^d}\frac{v(x)- v(y)}{|x-y|^{d+2s}}\,dy,&\quad\hbox{if}\,\,0<s<1/2,\\
{}\\
\displaystyle C_{d,s}\,\int_{\mathbb{R}^d}\frac{v(x)- v(y)-\nabla v(x)\cdotp(x-y)\mathcal{X}_{|x-y|<\delta}(y)}{|x-y|^{d+2s}}\,dy,\quad\forall\,\delta>0,&\quad\hbox{if}\,\,1/2\leq s<1,\\
\end{array}
\right.
\end{eqnarray*}
where $p.v.$ stands for Cauchy's principal value, and $C_{d,s}:= \frac{s\,4^s \Gamma(\frac{d}{2}+s)}{\pi^{\frac{d}{2}}\Gamma(1-s)}$.
\end{definition}

In fact, we are rarely going to use the fractional Laplacian operator in the Schwartz space; it can be extended to less regular functions as follows. For $s \in (0,1)$, $\varepsilon>0$, let
\begin{eqnarray*}
 \mathcal{L}_{s,\varepsilon}(\Omega)&:=&\left\{\begin{array}{ll}
\displaystyle L_s(\mathbb{R}^d)\cap C^{0,2s+\varepsilon}(\Omega)&\quad\hbox{if}\,\,0<s<1/2,\\
{}\\
\displaystyle L_s(\mathbb{R}^d)\cap C^{1,2s+\varepsilon-1}(\Omega),&\quad\hbox{if}\,\,1/2\leq s<1,\\
\end{array}
\right.
\end{eqnarray*}
where $\Omega$ be an open subset of $\mathbb{R}^d$, $C^{0,2s+\varepsilon}(\Omega)$ is the space of $2s+\varepsilon$- H\"{o}lder continuous functions on $\Omega$,  $C^{1,2s+\varepsilon-1}(\Omega)$ the space of functions of $C^1(\Omega)$ whose first partial
derivatives are H\"{o}lder continuous with exponent $2s+\varepsilon-1$, and
$$L_s(\mathbb{R}^d)=\left\{u:\mathbb{R}^d\rightarrow\mathbb{R}\quad\hbox{such that}\quad \int_{\mathbb{R}^d}\frac{u(x)}{1+|x|^{d+2s}}\,dx<\infty\right\}.$$

\begin{proposition}\label{Frac}\cite[Proposition~2.4]{Silvestre}${}$\\
Let $\Omega$ be an open subset of $\mathbb{R}^d$, $s \in (0,1)$, and $f\in \mathcal{L}_{s,\varepsilon}(\Omega)$ for some $\varepsilon>0$. Then $(-\Delta)^sf$ is a continuous function in $\Omega$ and $(-\Delta)^sf(x)$ is given by the pointwise formulas of Definition \ref{def1} for every $x\in\Omega$.
\end{proposition} 
\noindent{\bf Remark:} A simple sufficient condition for function $f$ to satisfy the conditions in Proposition \ref{Frac} is that $f\in  L^1_{loc}(\mathbb{R}^d)\cap C^{2}(\Omega)$.\\

\begin{lemma}\label{lemma3}\cite{Bonforte}
Let $\langle x\rangle:=(1+|x|^2)^{1/2}$, $x\in\mathbb{R}^d$, $s \in (0,1]$, $d\geq 1$, and $q_0>d$. Then 
$$\langle x\rangle^{-q_0}\in  L^\infty(\mathbb{R}^d)\cap C^{\infty}(\mathbb{R}^d),\qquad\partial_x^2\langle x\rangle^{-q_0}\in L^\infty(\mathbb{R}^d),$$
 and 
$$
\left|(-\Delta)^s\langle x\rangle^{-q_0}\right|\lesssim \langle x\rangle^{-d-2s}.
$$
\end{lemma}

\begin{lemma}\label{lemma4}
Let $\psi$ be a smooth function satisfying $\partial_x^2\psi\in L^\infty(\mathbb{R}^d)$. For any $R>0$, let $\psi_R$ be a function defined by
$$ \psi_R(x):= \psi(R^{-1} x) \quad \text{ for all } x \in \mathbb{R}^d.$$
Then, $(-\Delta)^s (\psi_R)$,  $s \in (0,1]$,  satisfies the following scaling properties:
$$(-\Delta)^s \psi_R(x)= R^{-2s}(-\Delta)^s\psi(R^{-1} x), \quad \text{ for all } x \in \mathbb{R}^d. $$
\end{lemma}

\begin{lemma}\label{lemma5} Let $R>0$, $p>1$, $0<s\leq1$, $d\geq1$, and $d<q_0<d+2sp$. Then, the following estimate holds
\begin{equation}\label{10}
\int_{\mathbb{R}^d}(\Phi_R(x))^{-1/(p-1)}\,\big|(-\Delta)^s\Phi_R(x)\big|^{p/(p-1)}\, dx\lesssim R^{-\frac{2sp}{p-1}+d},
\end{equation}
where $\Phi_R(x)=\langle {x}/{R}\rangle^{-q_0}=(1+|x/R|^2)^{-q_0/2}$.
\end{lemma}
\begin{proof}  Let $\tilde{x}=x/R$; by Lemma \ref{lemma4} we have $(-\Delta)^s\Phi_R(x)=R^{-2s}(-\Delta)^s\Phi_R(\tilde{x})$. Therefore, using Lemma \ref{lemma3}, we conclude that
\begin{eqnarray*}
&{}&\int_{\mathbb{R}^d}(\Phi_R(x))^{-1/(p-1)}\,\big|(-\Delta)^s \Phi_R(x)\big|^{p/(p-1)}\, dx\\
&{}&\lesssim R^{-\frac{2sp}{p-1}+d}\int_{\mathbb{R}^d}\langle \tilde{x}\rangle^{\frac{q_0}{p-1}-\frac{(d+2s)p}{p-1}}\, d \tilde{x}\\
&{}&\lesssim R^{-\frac{2sp}{p-1}+d},
\end{eqnarray*}
where we have used the fact that $q_0<d+2sp \Longleftrightarrow \frac{(d+2s)p}{p-1}-\frac{q_0}{p-1}>d$.
\end{proof}

\begin{lemma}\label{lemma6}
Let $0<\gamma,\theta\leq1$, $0<\mu,\sigma\leq2$, $p,q>1$, and $d\geq1$. Assume that $\theta p+\gamma pq-pq+1>0$, then
$$d\leq\overline{D}\,\Longleftrightarrow\,d\leq\max_{d_0>0}\min_{1\leq i\leq 4}h_i(d_0),$$
where $\overline{D}$ is defined in Theorem \ref{theo2}, and
$$\begin{array}{ll}
h_1(d_0):=\frac{p(\theta+\gamma q)-pq+1}{d_0(pq-1)},&\qquad h_2(d_0):=\frac{p(\theta+d_0\mu q)-pq+1}{d_0(pq-1)},\\\\
h_3(d_0):=\frac{p(\sigma d_0+\gamma q)-pq+1}{d_0(pq-1)},&\qquad h_4(d_0):=\frac{p(\sigma+\mu q)}{pq-1}-\frac{1}{d_0}.
\end{array}$$
\end{lemma}
\begin{proof} 
We have two cases:\\

\noindent{\bf Case I: $\displaystyle\frac{\theta}{\sigma}\leq \frac{\gamma}{\mu}$.} Then $\displaystyle d\leq\max_{d_0>0}\min_{1\leq i\leq 4}h_i(d_0)$ is equivalent to
\begin{eqnarray}\label{T21}
d&\leq&\max\left\{\max_{d_0\in]0,\frac{\theta}{\sigma}]}h_4(d_0);\max_{d_0\in[\frac{\theta}{\sigma},\frac{\gamma}{\mu}]}h_2(d_0);\max_{d_0\in[\frac{\gamma}{\mu},+\infty[}h_1(d_0)\right\}\nonumber\\
&=&\max\left\{h_4\left(\frac{\theta}{\sigma}\right);\max_{d_0\in[\frac{\theta}{\sigma},\frac{\gamma}{\mu}]}h_2(d_0);h_1\left(\frac{\gamma}{\mu}\right)\right\}.
\end{eqnarray}
If $p\leq\frac{1}{q-\theta}$, as $h_2(d_0)$ is a non-increasing function of $d_0$, \eqref{T21} is equivalent to
\begin{equation}\label{T22}
d\leq\max\left\{h_4\left(\frac{\theta}{\sigma}\right);h_2\left(\frac{\theta}{\sigma}\right);h_1\left(\frac{\gamma}{\mu}\right)\right\}
=h_2\left(\frac{\theta}{\sigma}\right)=\frac{\theta\sigma p+\theta\mu pq-\sigma pq+\sigma}{\theta(pq-1)}
=: D_1,
\end{equation}
because $h_2\left(\frac{\theta}{\sigma}\right)=h_4\left(\frac{\theta}{\sigma}\right)$ and $h_2\left(\frac{\theta}{\sigma}\right)\geq h_2\left(\frac{\gamma}{\mu}\right)=h_1\left(\frac{\gamma}{\mu}\right)$.\\

If $p\geq\frac{1}{q-\theta}$, as $h_2(d_0)$ is a non-decreasing function of $d_0$, \eqref{T21} is equivalent to
\begin{equation}\label{T23}
d\leq\max\left\{h_4\left(\frac{\theta}{\sigma}\right);h_2\left(\frac{\gamma}{\mu}\right);h_1\left(\frac{\gamma}{\mu}\right)\right\}
=h_2\left(\frac{\gamma}{\mu}\right)
=\frac{\mu(\theta p+\gamma pq- pq+1)}{\gamma(pq-1)}=: D_2,
\end{equation}
because $h_2\left(\frac{\gamma}{\mu}\right)=h_1\left(\frac{\gamma}{\mu}\right)$ and $h_2\left(\frac{\gamma}{\mu}\right)\geq h_2\left(\frac{\theta}{\sigma}\right)$.\\
Summarizing, \eqref{T21} is equivalent to
$$d\leq\max\{D_1,D_2\}.$$

\noindent{\bf Case II: $\displaystyle\frac{\theta}{\sigma}\geq \frac{\gamma}{\mu}$.} Then $\displaystyle d\leq\max_{d_0>0}\min_{1\leq i\leq 4}h_i(d_0)$ is equivalent to
\begin{eqnarray}\label{T24}
d&\leq&\max\left\{\max_{d_0\in]0,\frac{\gamma}{\mu}]}h_4(d_0);\max_{d_0\in[\frac{\gamma}{\mu},\frac{\theta}{\sigma}]}h_3(d_0);\max_{d_0\in[\frac{\theta}{\sigma},+\infty[}h_1(d_0)\right\}\nonumber\\
&=&\max\left\{h_4\left(\frac{\gamma}{\mu}\right);\max_{d_0\in[\frac{\theta}{\sigma},\frac{\gamma}{\mu}]}h_3(d_0);h_1\left(\frac{\theta}{\sigma}\right)\right\}.
\end{eqnarray}
If $(1-\gamma)pq\leq1$, as $h_3(d_0)$ is a non-increasing function of $d_0$, \eqref{T24} is equivalent to
\begin{equation}\label{T25}
d\leq\max\left\{h_4\left(\frac{\gamma}{\mu}\right);h_3\left(\frac{\gamma}{\mu}\right);h_1\left(\frac{\theta}{\sigma}\right)\right\}=h_3\left(\frac{\gamma}{\mu}\right)=\frac{\gamma\sigma p+\gamma\mu pq-\mu pq+\mu}{\gamma(pq-1)}
=: D_3,
\end{equation}
because $h_3\left(\frac{\gamma}{\mu}\right)=h_4\left(\frac{\gamma}{\mu}\right)$ and $h_3\left(\frac{\gamma}{\mu}\right)\geq h_3\left(\frac{\theta}{\sigma}\right)=h_1\left(\frac{\theta}{\sigma}\right)$.\\

If $(1-\gamma)pq\geq1$, as $h_3(d_0)$ is a non-decreasing function of $d_0$, \eqref{T24} is equivalent to
\begin{equation}\label{T26}
d\leq\max\left\{h_4\left(\frac{\gamma}{\mu}\right);h_3\left(\frac{\theta}{\sigma}\right);h_1\left(\frac{\theta}{\sigma}\right)\right\}=h_3\left(\frac{\theta}{\sigma}\right)=\frac{\sigma(\theta p+\gamma pq- pq+1)}{\theta(pq-1)}=: D_4,
\end{equation}
because $h_3\left(\frac{\theta}{\sigma}\right)\geq h_3\left(\frac{\gamma}{\mu}\right)=h_4\left(\frac{\gamma}{\mu}\right)$ and $ h_3\left(\frac{\theta}{\sigma}\right)=h_1\left(\frac{\theta}{\sigma}\right)$.\\
We can summarize our calculation by concluding that \eqref{T24} is equivalent to
$$d\leq\max\{D_3,D_4\}.$$

As a conclusion, Cases I and II can be summarized by:
$$d\leq\min\left\{\max\{D_1,D_2\};\max\{D_3,D_4\}\right\}=\overline{D}.$$
\end{proof}

\begin{lemma}\label{lemma7}
Let $0<\gamma,\theta\leq1$, $0<\mu,\sigma\leq2$, $p,q>1$, and $d\geq1$. Assume that $\gamma q+\theta pq-pq+1>0$, then
$$d\leq\overline{E}\,\Longleftrightarrow\,d\leq\max_{d_0>0}\min_{1\leq i\leq 4}H_i(d_0),$$
where $\overline{E}$ is defined in Theorem \ref{theo2}, and
$$\begin{array}{ll}
H_1(d_0):=\frac{p(\theta+\gamma q)-pq+1}{d_0(pq-1)},&\qquad H_2(d_0):=\frac{p(\theta+d_0\mu q)-pq+1}{d_0(pq-1)},\\\\
H_3(d_0):=\frac{p(\sigma d_0+\gamma q)-pq+1}{d_0(pq-1)},&\qquad H_4(d_0):=\frac{p(\sigma+\mu q)}{pq-1}-\frac{1}{d_0}.
\end{array}$$
\end{lemma}
\begin{proof} 
We have two cases:\\
\noindent{\bf Case I: $\displaystyle\frac{\theta}{\sigma}\leq \frac{\gamma}{\mu}$.} Then $\displaystyle d\leq\max_{d_0>0}\min_{1\leq i\leq 4}H_i(d_0)$ is equivalent to
\begin{eqnarray}\label{T29}
d&\leq&\max\left\{\max_{d_0\in]0,\frac{\theta}{\sigma}]}H_4(d_0);\max_{d_0\in[\frac{\theta}{\sigma},\frac{\gamma}{\mu}]}H_3(d_0);\max_{d_0\in[\frac{\gamma}{\mu},+\infty[}H_1(d_0)\right\}\nonumber\\
&=&\max\left\{H_4\left(\frac{\theta}{\sigma}\right);\max_{d_0\in[\frac{\theta}{\sigma},\frac{\gamma}{\mu}]}H_3(d_0);H_1\left(\frac{\gamma}{\mu}\right)\right\}.
\end{eqnarray}
If $(1-\theta)pq\leq1$, as $H_3(d_0)$ is a non-increasing function of $d_0$, \eqref{T29} is equivalent to
\begin{equation}\label{T30}
d\leq\max\left\{H_4\left(\frac{\theta}{\sigma}\right);H_3\left(\frac{\theta}{\sigma}\right);H_1\left(\frac{\gamma}{\mu}\right)\right\}=H_3\left(\frac{\theta}{\sigma}\right)=\frac{\theta\mu q+\theta\sigma pq-\sigma pq+\sigma}{\theta(pq-1)}=: E_1,
\end{equation}
because $H_3\left(\frac{\theta}{\sigma}\right)=H_4\left(\frac{\theta}{\sigma}\right)$ and $H_3\left(\frac{\theta}{\sigma}\right)\geq H_3\left(\frac{\gamma}{\mu}\right)=H_1\left(\frac{\gamma}{\mu}\right)$.\\

If $(1-\theta)pq\geq1$, as $H_3(d_0)$ is a non-decreasing function of $d_0$, \eqref{T29} is equivalent to
\begin{equation}\label{T31}
d\leq\max\left\{H_4\left(\frac{\theta}{\sigma}\right);H_3\left(\frac{\gamma}{\mu}\right);H_1\left(\frac{\gamma}{\mu}\right)\right\}=H_3\left(\frac{\gamma}{\mu}\right)=\frac{\mu(\gamma q+\theta pq- pq+1)}{\gamma(pq-1)}=: E_2,
\end{equation}
because $H_3\left(\frac{\gamma}{\mu}\right)=H_1\left(\frac{\gamma}{\mu}\right)$ and $H_3\left(\frac{\gamma}{\mu}\right)\geq H_3\left(\frac{\theta}{\sigma}\right)=H_4\left(\frac{\theta}{\sigma}\right)$.\\
Summarizing, \eqref{T29} is equivalent to
$$d\leq\max\{E_1,E_2\}.$$

\noindent{\bf Case II: $\displaystyle\frac{\theta}{\sigma}\geq \frac{\gamma}{\mu}$.} Then $\displaystyle d\leq\max_{d_0>0}\min_{1\leq i\leq 4}h_i(d_0)$ is equivalent to
\begin{eqnarray}\label{T32}
d&\leq&\max\left\{\max_{d_0\in]0,\frac{\gamma}{\mu}]}H_4(d_0);\max_{d_0\in[\frac{\gamma}{\mu},\frac{\theta}{\sigma}]}H_2(d_0);\max_{d_0\in[\frac{\theta}{\sigma},+\infty[}H_1(d_0)\right\}\nonumber\\
&=&\max\left\{H_4\left(\frac{\gamma}{\mu}\right);\max_{d_0\in[\frac{\theta}{\sigma},\frac{\gamma}{\mu}]}H_2(d_0);H_1\left(\frac{\theta}{\sigma}\right)\right\}.
\end{eqnarray}
If $q\leq\frac{1}{p-\gamma}$, as $H_2(d_0)$ is a non-increasing function of $d_0$, \eqref{T32} is equivalent to
\begin{equation}\label{T33}
d\leq\max\left\{H_4\left(\frac{\gamma}{\mu}\right);H_2\left(\frac{\gamma}{\mu}\right);H_1\left(\frac{\theta}{\sigma}\right)\right\}=H_2\left(\frac{\gamma}{\mu}\right)=\frac{\gamma\mu q+\gamma\sigma pq-\mu pq+\mu}{\gamma(pq-1)}=: E_3,
\end{equation}
because $H_2\left(\frac{\gamma}{\mu}\right)=H_4\left(\frac{\gamma}{\mu}\right)$ and $H_2\left(\frac{\gamma}{\mu}\right)\geq H_2\left(\frac{\theta}{\sigma}\right)=H_1\left(\frac{\theta}{\sigma}\right)$.\\

If $q\geq\frac{1}{p-\gamma}$, as $H_2(d_0)$ is a non-decreasing function of $d_0$, \eqref{T32} is equivalent to
\begin{equation}\label{T34}
d\leq\max\left\{H_4\left(\frac{\gamma}{\mu}\right);H_2\left(\frac{\theta}{\sigma}\right);H_1\left(\frac{\theta}{\sigma}\right)\right\}=H_2\left(\frac{\theta}{\sigma}\right)=\frac{\sigma(\gamma q+\theta pq- pq+1)}{\theta(pq-1)}=: E_4,
\end{equation}
because $H_2\left(\frac{\theta}{\sigma}\right)=H_1\left(\frac{\theta}{\sigma}\right)$ and $H_2\left(\frac{\theta}{\sigma}\right)\geq H_2\left(\frac{\gamma}{\mu}\right)=H_4\left(\frac{\gamma}{\mu}\right)$.\\
Consequently, \eqref{T32} is equivalent to
$$d\leq\max\{E_3,E_4\}.$$

So, Cases I and II lead to
$$d\leq\min\left\{\max\{E_1,E_2\};\max\{E_3,E_4\}\right\}=\overline{E}.$$
\end{proof}


\section{Proof of Theorems \ref{theo1} and \ref{theo2}}\label{sec4}

\noindent {\bf Proof of Theorem \ref{theo1}}. Let $u$ be a global nontrivial weak solution of \eqref{1}. Then
\begin{eqnarray}\label{weaksolution1}
&{}&\int_{Q_T}|u(t,g(x))|^p\varphi(t,x)\,dt\,dx+\int_{Q_T}u_0(x)\,D^\alpha_{t|T}\varphi(t,x)\,dt\,dx\nonumber\\
&{}&=\int_{Q_T}u(t,x)\,D^\alpha_{t|T}\varphi(t,x)\,dt\,dx+\int_{Q_T}u(t,x)(-\Delta)^{\delta/2}\varphi(t,x)\,dt\,dx,
\end{eqnarray}
for all $\varphi\in X_{\delta,T}$ and all $T>0$. By introducing $\varphi^{1/p}\varphi^{-1/p}$ and applying the following Young's inequality
\[
AB\leq\frac{c_0}{4}A^p+C(p,c_0)B^{p^\prime},\quad A\geq0,\;B\geq0,\;p+p^\prime=p p^\prime,
\]
where $c_0$ is introduced in $(\textup{A}1)$, we get
\begin{equation}\label{T1}
\int_{Q_T}u(t,x)\,D^\alpha_{t|T}\varphi(t,x)\,dt\,dx\leq\frac{c_0}{4}\int_{Q_T}|u(t,x)|^{p}\varphi(t,x)\,dt\,dx+C\int_{Q_T}\varphi^{-\frac{1}{p-1}}(t,x) \left|D^{\alpha}_{t|T}\varphi(t,x)\right|^{p^\prime}\,dt\,dx,
\end{equation}
and
\begin{equation}\label{T2}
\int_{Q_T}u(t,x)(-\Delta)^{\delta/2}\varphi(t,x)\,dt\,dx\leq\frac{c_0}{4}\int_{Q_T}|u(t,x)|^{p}\varphi(t,x)\,dt\,dx+C\int_{Q_T}\varphi^{-\frac{1}{p-1}}(t,x) \left|(-\Delta)^{\delta/2}\varphi(t,x)\right|^{p^\prime}\,dt\,dx.
\end{equation}
Let
$$\varphi(x,t)=\Phi_R(x)\phi(t),$$
where $\Phi_R$ is defined in Lemma \ref{lemma3} for $R>0$, and $\phi$ is defined in \eqref{w}. Observe that, using 
$(\textup{A}1)$-$(\textup{A}2)$ and the monotonicity of $\Phi_R$, we obtain the estimate
\begin{eqnarray}\label{T3}
\int_{Q_T}|u(t,g(x))|^p\varphi(t,x)\,dt\,dx&=&\int_{Q_T}|u(t,x)|^p\phi(t)\Phi_R(g^{-1}(x))|J_g^{-1}(x)|\,dt\,dx\nonumber\\
&\geq& c_0\int_{Q_T}|u(t,x)|^p\phi(t)\Phi_R(x)\,dt\,dx
\end{eqnarray}
Using the estimates \eqref{T1}-\eqref{T3} into \eqref{weaksolution1} we get
\begin{eqnarray*}
&{}&\frac{c_0}{2}\int_{Q_T}|u(t,x)|^p\varphi(t,x)\,dt\,dx+\int_{Q_T}u_0(x)\,D^\alpha_{t|T}\varphi(t,x)\,dt\,dx\nonumber\\
&{}&\leq C\int_{Q_T}\varphi^{-\frac{1}{p-1}}(t,x) \left|D^{\alpha}_{t|T}\varphi(t,x)\right|^{p^\prime}\,dt\,dx+C\int_{Q_T}\varphi^{-\frac{1}{p-1}}(t,x) \left|(-\Delta)^{\delta/2}\varphi(t,x)\right|^{p^\prime}\,dt\,dx.
\end{eqnarray*}
Whereupon, using Lemmas \ref{lemma2} and \ref{lemma5}, we arrive at
\begin{eqnarray}\label{9}
\int_{Q_T}|u(t,x)|^p\varphi(t,x)\,dt\,dx&\lesssim& \int_{\mathbb{R}^d}|u_0(x)|\,dx\int_0^T D^\alpha_{t|T}\phi(t)\,dt\nonumber\\
&{}&+\int_{\mathbb{R}^d}\Phi_R(x)\,dx\int_0^T\phi^{-\frac{1}{p-1}}(t) \left|D^{\alpha}_{t|T}\phi(t,x)\right|^{p^\prime}\,dt\nonumber\\
&{}&+\int_0^T\phi(t)\,dt\int_{\mathbb{R}^d}\Phi_R^{-\frac{1}{p-1}}(x) \left|(-\Delta)^{\delta/2}\Phi_R(x)\right|^{p^\prime}\,dx\nonumber\\
&\lesssim& T^{-\alpha} \int_{\mathbb{R}^d}|u_0(x)|\,dx+\,R^d\,T^{1-\alpha\frac{p}{p-1}} +\,T\,R^{-\frac{\delta p}{p-1}+d}.
\end{eqnarray}
At this stage, two cases can be distinguished.\\
\noindent {\bf Case 1}: If $p<p_*$, we set $R:=T^{\alpha/\delta}$, then \eqref{9} implies
\begin{equation}\label{11}
\int_{Q_T}|u(t,x)|^p\varphi(t,x)\,dt\,dx\lesssim\,T^{-\alpha} \int_{\mathbb{R}^d}|u_0(x)|\,dx+\,T^{1-\alpha\frac{p}{p-1}+\frac{\alpha d}{\delta}}.
\end{equation}
Letting $T\rightarrow+\infty$, using the fact that $p<p_*\Longleftrightarrow 1-\alpha\frac{p}{p-1}+\frac{\alpha d}{\delta}<0$, the assumption $u_0 \in L^1$ and Lebesgue's dominated convergence theorem, we conclude that
\begin{equation}
\int_0^\infty\int_{\mathbb{R}^d}|u(t,x)|^p\,dt\,dx\leq 0,
\end{equation}
which leads to a contradiction.\\
\noindent {\bf Case 2}: If $p=p_*$ and $\alpha=1$. Let $\widetilde{\phi}$ be a smooth nonnegative non-increasing function such that  $0\leq \widetilde{\phi} \leq 1$ and
\[
\widetilde{\phi}(t)=\left\{\begin {array}{ll}\displaystyle{1}&\displaystyle{\quad\text{if }0\leq t\leq 1/2,}\\
{}\\
\displaystyle{0}&\displaystyle{\quad\text {if }t\geq 1.}
\end {array}\right.
\]
Using $\widetilde{\phi}^\ell(t)$, $\ell>p^{\prime}$, instead of $\phi(t)$ and applying H\"{o}lder's inequality instead of Young's inequality into \eqref{T1}, we obtain
\begin{eqnarray}\label{T4}
\int_{Q_T}u(t,x)\,\varphi_t(t,x)\,dt\,dx&\leq& \left(\int_{\widetilde{Q}_T}|u(t,x)|^{p}\varphi(t,x)\,dt\,dx\right)^{1/p}\left(\int_{Q_T}\varphi^{-\frac{1}{p-1}}(t,x) \left|\varphi_t(t,x)\right|^{p^\prime}\,dt\,dx\right)^{1/p^\prime}\nonumber\\
&=&\left(\int_{\widetilde{Q}_T}|u(t,x)|^{p}\varphi(t,x)\,dt\,dx\right)^{1/p}\left(\int_{\mathbb{R}^d}\Phi_R(x)\,dx\int_0^T\widetilde{\phi}^{\ell -p^\prime}(t)\left|\frac{d}{dt}\widetilde{\phi}(t)\right|^{p^\prime}\,dt\right)^{1/p^\prime}\nonumber\\
&\lesssim& T^{-1+\frac{1}{p^\prime}}\,R^{\frac{d}{p^\prime}}\left(\int_{\widetilde{Q}_T}|u(t,x)|^{p}\varphi(t,x)\,dt\,dx\right)^{1/p}
\end{eqnarray}
where $\widetilde{Q}_T=[T/2,T]\times \mathbb{R}^d$. We set $R:=K^{\alpha/\delta}T^{\alpha/\delta}$,  where $K\ge 1$ is independent of $T$. Using the estimates \eqref{T2}-\eqref{T3} and \eqref{T4} into \eqref{weaksolution1} and taking account of $p=p_*$, we get
\begin{eqnarray}\label{13}
\int_{Q_T}|u(t,x)|^p\varphi(t,x)\,dt\,dx &\lesssim& T^{-\alpha} \int_{\mathbb{R}^d}|u_0(x)|\,dx+\,T\,R^{-\frac{\delta p}
{p-1}+d}\nonumber\\
&{}&\, +\,T^{-1+\frac{1}{p^\prime}}\,R^{\frac{d}{p^\prime}}\left(\int_{\widetilde{Q}_T}|u(t,x)|^{p}\varphi(t,x)\,dt\,dx\right)^{1/p}\nonumber\\
&=&T^{-\alpha} \int_{\mathbb{R}^d}|u_0(x)|\,dx+\,K^{-1}\nonumber\\
&{}&\, +\,K^{\frac{d}{p^\prime}}\left(\int_{\widetilde{Q}_T}|u(t,x)|^{p}\varphi(t,x)\,dt\,dx\right)^{1/p}.
\end{eqnarray}
On the other hand, using \eqref{11} with $T\rightarrow\infty $ and taking account of $p=p_*,$ we obtain
 $$
u\in L^p((0,\infty),L^p(\mathbb{R}^d));
$$
consequently,
\begin{equation}\label{12}
\lim_{T\rightarrow\infty}\int_{\widetilde{Q}_T}|u(t,x)|^{p}\varphi(t,x)\,dt\,dx=\lim_{T\rightarrow\infty}\left(\int_{Q_T}|u(t,x)|^{p}\varphi(t,x)\,dt\,dx-\int_{Q_{T/2}}|u(t,x)|^{p}\varphi(t,x)\,dt\,dx\right)=0.
\end{equation}
Finally, letting $T\rightarrow+\infty$ into \eqref{13} and using \eqref{12} we arrive at
$$
\int_0^\infty\int_{\mathbb{R}^d}|u(t,x)|^p\,dt\,dx \lesssim K^{-1},
$$
which leads to a contradiction for $K\gg1$.\\
${}$\hfill$\blacksquare$\\


\noindent {\bf Proof of Theorem \ref{theo2}}. Let $(u,v)$ be a global nontrivial weak solution of \eqref{15}. Then
\begin{eqnarray}\label{weaksolution2}
&{}&\int_{Q_T}|v(t,g(x))|^p\varphi(t,x)\,dt\,dx+\int_{Q_T}u_0(x)\,D^\gamma_{t|T}\varphi(t,x)\,dt\,dx\nonumber\\
&{}&=\int_{Q_T}u(t,x)\,D^\gamma_{t|T}\varphi(t,x)\,dt\,dx+\int_{Q_T}u(t,x)(-\Delta)^{\mu/2}\varphi(t,x)\,dt\,dx,
\end{eqnarray}
and
\begin{eqnarray}\label{weaksolution3}
&{}&\int_{Q_T}|u(t,f(x))|^q\psi(t,x)\,dt\,dx+\int_{Q_T}v_0(x)\,D^\theta_{t|T}\psi(t,x)\,dt\,dx\nonumber\\
&{}&=\int_{Q_T}v(t,x)\,D^\theta_{t|T}\psi(t,x)\,dt\,dx+\int_{Q_T}v(t,x)(-\Delta)^{\sigma/2}\psi(t,x)\,dt\,dx,
\end{eqnarray}
hold for all $\varphi\in X_{\mu,T}$, $\psi\in X_{\sigma,T}$, and all $T>0$. By applying H\"{o}lder's inequality, we have
\begin{eqnarray}\label{T10}
&{}&\int_{Q_T}u(t,x)\,D^\gamma_{t|T}\varphi(t,x)\,dt\,dx\nonumber\\
&{}&\,=\int_{Q_T}u(t,x)\,\psi^{1/q}\psi^{-1/q}\,D^\gamma_{t|T}\varphi(t,x)\,dt\,dx\nonumber\\
&{}&\,\leq\left(\int_{Q_T}|u(t,x)|^{q}\psi(t,x)\,dt\,dx\right)^{1/q}\left(\int_{Q_T}\psi^{-\frac{1}{q-1}}(t,x) \left|D^{\gamma}_{t|T}\varphi(t,x)\right|^{q^\prime}\,dt\,dx\right)^{1/q^\prime}.
\end{eqnarray}
where $q^\prime=q/{(q-1)}$. Similarly, we obtain
\begin{eqnarray}\label{T11}
&{}&\int_{Q_T}v(t,x)\,D^\theta_{t|T}\psi(t,x)\,dt\,dx\nonumber\\
&{}&\,\leq\left(\int_{Q_T}|v(t,x)|^{p}\varphi(t,x)\,dt\,dx\right)^{1/p}\left(\int_{Q_T}\varphi^{-\frac{1}{p-1}}(t,x) \left|D^{\theta}_{t|T}\psi(t,x)\right|^{p^\prime}\,dt\,dx\right)^{1/p^\prime},
\end{eqnarray}
\begin{eqnarray}\label{T12}
&{}&\int_{Q_T}u(t,x)\,(-\Delta)^{\mu/2}\varphi(t,x)\,dt\,dx\nonumber\\
&{}&\,\leq\left(\int_{Q_T}|u(t,x)|^{q}\psi(t,x)\,dt\,dx\right)^{1/q}\left(\int_{Q_T}\psi^{-\frac{1}{q-1}}(t,x) \left|(-\Delta)^{\mu/2}\varphi(t,x)\right|^{q^\prime}\,dt\,dx\right)^{1/q^\prime},
\end{eqnarray}
and
\begin{eqnarray}\label{T13}
&{}&\int_{Q_T}v(t,x)\,(-\Delta)^{\sigma/2}\psi(t,x)\,dt\,dx\nonumber\\
&{}&\,\leq\left(\int_{Q_T}|v(t,x)|^{p}\varphi(t,x)\,dt\,dx\right)^{1/p}\left(\int_{Q_T}\varphi^{-\frac{1}{p-1}}(t,x) \left|(-\Delta)^{\sigma/2}\psi(t,x)\right|^{p^\prime}\,dt\,dx\right)^{1/p^\prime},
\end{eqnarray}
where $p^\prime=p/{(p-1)}$. Let
$$\varphi(t,x)=\psi(t,x)=\Phi_{R}(x)\phi(t),$$
where $\Phi_{R}$ is defined in Lemma \ref{lemma3} for $R>0$, and $\phi$ is defined in \eqref{w}. Using  $(\textup{A}1)$-$(\textup{A}2)$, we get
\begin{eqnarray}\label{T14}
\int_{Q_T}|v(t,g(x))|^p\varphi(t,x)\,dt\,dx&=&\int_{Q_T}|v(t,x)|^p\phi(t)\Phi_{R_1}(g^{-1}(x))|J_g^{-1}(x)|\,dt\,dx\nonumber\\
&\geq& c_0\int_{Q_T}|v(t,x)|^p\phi(t)\Phi_{R_1}(x)\,dt\,dx\nonumber\\
&=& c_0\int_{Q_T}|v(t,x)|^p\varphi(t,x)\,dt\,dx
\end{eqnarray}
and
\begin{eqnarray}\label{T15}
\int_{Q_T}|u(t,f(x))|^q\varphi(t,x)\,dt\,dx&=&\int_{Q_T}|u(t,x)|^q\phi(t)\Phi_{R_2}(f^{-1}(x))|J_f^{-1}(x)|\,dt\,dx\nonumber\\
&\geq& c_0\int_{Q_T}|u(t,x)|^q\phi(t)\Phi_{R_2}(x)\,dt\,dx\nonumber\\
&=& c_0\int_{Q_T}|u(t,x)|^q\varphi(t,x)\,dt\,dx
\end{eqnarray}
Using Lemma \ref{lemma2} and inserting the estimates \eqref{T10}-\eqref{T15} into \eqref{weaksolution3}, we get
\begin{equation}\label{T16}
c_0\int_{Q_T}|v(t,x)|^p\varphi(t,x)\,dt\,dx\leq \left(\int_{Q_T}|u(t,x)|^{q}\varphi(t,x)\,dt\,dx\right)^{1/q}\mathcal{A}+\,C\,T^{-\gamma}\int_{\mathbb{R}^d}|u_0(x)|\,dx,
\end{equation}
and
\begin{equation}\label{T17}
c_0\int_{Q_T}|u(t,x)|^q\varphi(t,x)\,dt\,dx\leq \left(\int_{Q_T}|v(t,x)|^{p}\varphi(t,x)\,dt\,dx\right)^{1/p}\mathcal{B}+\,C\,T^{-\theta}\int_{\mathbb{R}^d}|v_0(x)|\,dx,
\end{equation}
where
$$\mathcal{A}:=\left(\int_{Q_T}\varphi^{-\frac{1}{q-1}}(t,x) \left|D^{\gamma}_{t|T}\varphi(t,x)\right|^{q^\prime}\,dt\,dx\right)^{1/q^\prime}+\left(\int_{Q_T}\varphi^{-\frac{1}{q-1}}(t,x) \left|(-\Delta)^{\mu/2}\varphi(t,x)\right|^{q^\prime}\,dt\,dx\right)^{1/q^\prime}$$
and
$$\mathcal{B}:=\left(\int_{Q_T}\varphi^{-\frac{1}{p-1}}(t,x) \left|D^{\theta}_{t|T}\varphi(t,x)\right|^{p^\prime}\,dt\,dx\right)^{1/p^\prime}+\left(\int_{Q_T}\varphi^{-\frac{1}{p-1}}(t,x) \left|(-\Delta)^{\sigma/2}\varphi(t,x)\right|^{p^\prime}\,dt\,dx\right)^{1/p^\prime}.$$
Now, combining the terms in \eqref{T16}-\eqref{T17}, we arrive at
\begin{equation}\label{T18}
\int_{Q_T}|v(t,x)|^p\varphi(t,x)\,dt\,dx\lesssim \mathcal{A}^{\frac{pq}{pq-1}}\,\mathcal{B}^{\frac{p}{pq-1}}+\,T^{-\frac{\theta}{q}}\,\mathcal{A}\left(\int_{\mathbb{R}^d}|v_0(x)|\,dx\right)^{1/q}+\,T^{-\gamma}\int_{\mathbb{R}^d}|u_0(x)|\,dx,
\end{equation}
and
\begin{equation}\label{T19}
\int_{Q_T}|u(t,x)|^q\varphi(t,x)\,dt\,dx\lesssim \mathcal{A}^{\frac{q}{pq-1}}\,\mathcal{B}^{\frac{pq}{pq-1}}+\,T^{-\frac{\gamma}{p}}\,\mathcal{B}\left(\int_{\mathbb{R}^d}|u_0(x)|\,dx\right)^{1/p}+\,T^{-\theta}\int_{\mathbb{R}^d}|v_0(x)|\,dx.
\end{equation}
At this stage, we distinguish five cases:\\
\noindent{\bf Case 1: $\displaystyle d<\overline{D}$}.\\
 In this case, we take $R=T^{d_0}$, $d_0>0$. Then
\begin{equation}\label{T47}
\mathcal{A}^{\frac{pq}{pq-1}}\,\mathcal{B}^{\frac{p}{pq-1}}\lesssim T^{\sigma_1}+\,T^{\sigma_2}+\,T^{\sigma_3}+\,T^{\sigma_4},
\end{equation}
where
$$\begin{array}{ll}
\sigma_1=d_0d+1-\frac{p(\sigma d_0+\gamma q)}{pq-1},&\qquad \sigma_2=d_0d+1-\frac{d_0p(\sigma+\mu q)}{pq-1},\\\\
\sigma_3=d_0d+1-\frac{p(\theta+\gamma q)}{pq-1},&\qquad \sigma_4=d_0d+1-\frac{p(\theta+d_0\mu q)}{pq-1}.
\end{array}$$
In order to have all exponents of $T$ negative, it is sufficient to require $\sigma_i<0$, $1\leq i\leq 4$, which in tun is equivalent to
$$d<\max_{d_0>0}\min_{1\leq i\leq 4}h_i(d_0),$$
where
$$\begin{array}{ll}
h_1(d_0):=\frac{p(\theta+\gamma q)-pq+1}{d_0(pq-1)},&\qquad h_2(d_0):=\frac{p(\theta+d_0\mu q)-pq+1}{d_0(pq-1)},\\\\
h_3(d_0):=\frac{p(\sigma d_0+\gamma q)-pq+1}{d_0(pq-1)},&\qquad h_4(d_0):=\frac{p(\sigma+\mu q)}{pq-1}-\frac{1}{d_0},
\end{array}$$
 and this is true due to Lemma \ref{lemma6}. Therefore
 \begin{equation}\label{T37}
 \lim_{T\rightarrow\infty}\mathcal{A}^{\frac{pq}{pq-1}}\,\mathcal{B}^{\frac{p}{pq-1}}=0.
 \end{equation}
 In addition, we have
\begin{equation}\label{T48}
T^{-\frac{\theta}{q}}\,\mathcal{A}\lesssim T^{-\frac{\theta}{q}+\frac{d_0d(q-1)}{q}+\frac{q-1}{q}-\gamma}+T^{-\frac{\theta}{q}+\frac{d_0d(q-1)}{q}+\frac{q-1}{q}-\mu d_0}.
 \end{equation}
Using the fact that $\sigma_3,\sigma_4<0$, we see that 
$$-\frac{\theta}{q}+\frac{d_0d(q-1)}{q}+\frac{q-1}{q}-\gamma<\frac{(\theta+\gamma q)(1-p)}{q(pq-1)}<0,$$
and
$$-\frac{\theta}{q}+\frac{d_0d(q-1)}{q}+\frac{q-1}{q}-\mu d_0<\frac{(\theta+\mu d_0q)(1-p)}{q(pq-1)}<0,$$
which implies that 
\begin{equation}\label{T38}
\lim_{T\rightarrow\infty}T^{-\frac{\theta}{q}}\,\mathcal{A}=0.
\end{equation}
Taking to the limit when $T\rightarrow\infty$ in $(\ref{T18})$ and using \eqref{T37}-\eqref{T38} and $u_0\in L^1(\mathbb{R}^d)$, we get
$$\lim_{T\rightarrow\infty}\int_{Q_T}|v(t,x)|^p\varphi(t,x)\,dt\,dx =0,$$
which implies by the monotone convergence theorem that
$$\int_0^\infty\int_{\mathbb{R}^d}|v(t,x)|^p\,dx\,dt =0,$$
and then $v\equiv0$ a.e.. Thus, combing it with \eqref{T17}, we derive that 
$$\int_{Q_T}|u(t,x)|^q\varphi(t,x)\,dt\,dx\lesssim T^{-\theta}\int_{\mathbb{R}^d}|v_0(x)|\,dx,$$
which yields, by the monotone convergence theorem and $v_0\in L^1(\mathbb{R}^d)$, that
$$\lim_{T\rightarrow\infty}\int_{Q_T}|u(t,x)|^q\varphi(t,x)\,dt\,dx=0,$$
i.e. $v\equiv0$ a.e.; contradiction.\\

\noindent{\bf Case 2: $\displaystyle d<\overline{E}$}.\\
 In this case, we take $R=T^{d_0}$, $d_0>0$. Then, we have
\begin{equation}\label{T55}
\mathcal{A}^{\frac{q}{pq-1}}\,\mathcal{B}^{\frac{pq}{pq-1}}\lesssim T^{\rho_1}+\,T^{\rho_2}+\,T^ {\rho_3}+\,T^{\rho_4},
\end{equation}
where
$$\begin{array}{ll}
\rho_1=d_0d+1-\frac{q(\gamma+\theta p)}{pq-1},&\qquad \rho_2=d_0d+1-\frac{q(\gamma+d_0\sigma p)}{pq-1},\\\\
\rho_3=d_0d+1-\frac{q(\mu d_0+\theta p)}{pq-1},&\qquad \rho_4=d_0d+1-\frac{d_0q(\mu+\sigma p)}{pq-1}.
\end{array}$$
Therefore, in order to have all exponents of $T$ negative, it is sufficient to require $\rho_i<0$, $1\leq i\leq 4$, which is equivalent to
$$
d<\max_{d_0>0}\min_{1\leq i\leq 4}H_i(d_0),
$$
where
$$\begin{array}{ll}
H_1(d_0)=\frac{q(\gamma+\theta p)-pq+1}{d_0(pq-1)},&\qquad H_2(d_0)=\frac{q(\gamma+d_0\sigma p)-pq+1}{d_0(pq-1)},\\\\
H_3(d_0)=\frac{q(\mu d_0+\theta p)-pq+1}{d_0(pq-1)},&\qquad H_4(d_0)=\frac{q(\mu+\sigma p)}{pq-1}-\frac{1}{d_0},
\end{array}$$
and this is true due to Lemma \ref{lemma7}. Therefore
 \begin{equation}\label{T39}
 \lim_{T\rightarrow\infty}\mathcal{A}^{\frac{q}{pq-1}}\,\mathcal{B}^{\frac{pq}{pq-1}}=0.
 \end{equation}
 In addition, we have
\begin{equation}\label{T56}
T^{-\frac{\gamma}{p}}\,\mathcal{B}\lesssim T^{-\frac{\gamma}{p}+\frac{d_0d(p-1)}{p}+\frac{p-1}{p}-\theta}+T^{-\frac{\gamma}{p}+\frac{d_0d(p-1)}{p}+\frac{p-1}{p}-\sigma d_0}.
\end{equation}
Using the fact that $\rho_1,\rho_2<0$, we can see that 
$$-\frac{\gamma}{p}+\frac{d_0d(p-1)}{p}+\frac{p-1}{p}-\theta<\frac{(\gamma+\theta p)(1-q)}{p(pq-1)}<0,$$
and
$$-\frac{\gamma}{p}+\frac{d_0d(p-1)}{p}+\frac{p-1}{p}-\sigma d_0<\frac{(\gamma+\sigma d_0p)(1-q)}{p(pq-1)}<0,$$

which implies that 
\begin{equation}\label{T40}
\lim_{T\rightarrow\infty}T^{-\frac{\gamma}{p}}\,\mathcal{B}=0.
\end{equation}
Taking to the limit when $T\rightarrow\infty$ in $(\ref{T19})$ and using \eqref{T39}-\eqref{T40} and $v_0\in L^1(\mathbb{R}^d)$, we get
$$\lim_{T\rightarrow\infty}\int_{Q_T}|u(t,x)|^q\varphi(t,x)\,dt\,dx =0,$$
which implies by the monotone convergence theorem that
$$\int_0^\infty\int_{\mathbb{R}^d}|u(t,x)|^q\,dx\,dt =0,$$
and then $u\equiv0$ a.e.. Thus, combing it with \eqref{T16}, we derive that 
$$\int_{Q_T}|v(t,x)|^p\varphi(t,x)\,dt\,dx\lesssim T^{-\gamma}\int_{\mathbb{R}^d}|u_0(x)|\,dx,$$
which yields, by the monotone convergence theorem and $u_0\in L^1(\mathbb{R}^d)$, that
$$\lim_{T\rightarrow\infty}\int_{Q_T}|v(t,x)|^p\varphi(t,x)\,dt\,dx=0,$$
i.e. $v\equiv0$ a.e.; contradiction.\\

\noindent{\bf Case 3: $\displaystyle d=\overline{D}$, $\theta=1$, and $\gamma\neq1$}.\\
  In this case, we also suppose that $pq>1/(1-\gamma)$. Let $\widetilde{\phi}$ be a smooth nonnegative non-increasing function such that  $0\leq \widetilde{\phi} \leq 1$ and
\[
\widetilde{\phi}(t)=\left\{\begin {array}{ll}\displaystyle{1}&\displaystyle{\quad\text{if }0\leq t\leq 1/2,}\\
{}\\
\displaystyle{0}&\displaystyle{\quad\text {if }t\geq 1.}
\end {array}\right.
\]
Using $\theta=1$ and $\widetilde{\phi}^\ell(t)$, $\ell>p^{\prime}$, instead of $\phi(t)$, we may improve \eqref{T11} as follows
 \begin{equation}\label{T41}
\int_{Q_T}v(t,x)\,\varphi_t(t,x)\,dt\,dx\leq\left(\int_{\widetilde{Q}_T}|v(t,x)|^{p}\varphi(t,x)\,dt\,dx\right)^{1/p}\left(\int_{Q_T}\varphi^{-\frac{1}{p-1}}(t,x) \left|\varphi_t(t,x)\right|^{p^\prime}\,dt\,dx\right)^{1/p^\prime},
\end{equation}
where $\widetilde{Q}_T=[T/2,T]\times \mathbb{R}^d$, and so \eqref{T17} becomes
\begin{eqnarray}\label{T42}
c_0\int_{Q_T}|u(t,x)|^q\varphi(t,x)\,dt\,dx&\leq& \left(\int_{\widetilde{Q}_T}|v(t,x)|^{p}\varphi(t,x)\,dt\,dx\right)^{1/p}\left(\int_{Q_T}\varphi^{-\frac{1}{p-1}}(t,x) \left|\varphi_t(t,x)\right|^{p^\prime}\,dt\,dx\right)^{1/p^\prime}\nonumber\\
&{}&+\,\left(\int_{Q_T}|v(t,x)|^{p}\varphi(t,x)\,dt\,dx\right)^{1/p}\left(\int_{Q_T}\varphi^{-\frac{1}{p-1}}(t,x) \left|(-\Delta)^{\sigma/2}\varphi(t,x)\right|^{p^\prime}\,dt\,dx\right)^{1/p^\prime}\nonumber\\
&{}&+\,C\,T^{-\theta}\int_{\mathbb{R}^d}|v_0(x)|\,dx,
\end{eqnarray}
Now, inserting \eqref{T42} into \eqref{T16}, we arrive at
\begin{eqnarray}\label{T43}
\int_{Q_T}|v(t,x)|^p\varphi(t,x)\,dt\,dx&\lesssim&\mathcal{A}\left(\int_{\widetilde{Q}_T}|v(t,x)|^{p}\varphi(t,x)\,dt\,dx\right)^{1/qp}\left(\int_{Q_T}\varphi^{-\frac{1}{p-1}}(t,x) \left|\varphi_t(t,x)\right|^{p^\prime}\,dt\,dx\right)^{1/qp^\prime}\nonumber\\
&{}&+\,\left(\int_{Q_T}|v(t,x)|^{p}\varphi(t,x)\,dt\,dx\right)^{1/qp}\mathcal{A}\left(\int_{Q_T}\varphi^{-\frac{1}{p-1}}(t,x) \left|(-\Delta)^{\sigma/2}\varphi(t,x)\right|^{p^\prime}\,dt\,dx\right)^{1/qp^\prime}\nonumber\\
&{}&+\,\,T^{-\frac{\theta}{q}}\mathcal{A}\left(\int_{\mathbb{R}^d}|v_0(x)|\,dx\right)^{1/q}+\,T^{-\gamma}\int_{\mathbb{R}^d}|u_0(x)|\,dx.
\end{eqnarray}
Applying the following Young's inequality
$$
AB\leq\frac{1}{2}A^{pq}+C(p,q)B^{\frac{pq}{pq-1}},\quad A\geq0,\;B\geq0,
$$
on the second term in the right-hand side of \eqref{T43}, we arrive at
\begin{eqnarray}\label{T44}
\frac{1}{2}\int_{Q_T}|v(t,x)|^p\varphi(t,x)\,dt\,dx&\lesssim&\mathcal{A}\left(\int_{\widetilde{Q}_T}|v(t,x)|^{p}\varphi(t,x)\,dt\,dx\right)^{1/qp}\left(\int_{Q_T}\varphi^{-\frac{1}{p-1}}(t,x) \left|\varphi_t(t,x)\right|^{p^\prime}\,dt\,dx\right)^{1/qp^\prime}\nonumber\\
&{}&+\,\mathcal{A}^{\frac{pq}{pq-1}}\left(\int_{Q_T}\varphi^{-\frac{1}{p-1}}(t,x) \left|(-\Delta)^{\sigma/2}\varphi(t,x)\right|^{p^\prime}\,dt\,dx\right)^{\frac{p-1}{pq-1}}\nonumber\\
&{}&+\,T^{-\frac{\theta}{q}}\mathcal{A}\left(\int_{\mathbb{R}^d}|v_0(x)|\,dx\right)^{1/q}+\,T^{-\gamma}\int_{\mathbb{R}^d}|u_0(x)|\,dx.
\end{eqnarray}
At this stage, we set $R=K^{d_0}T^{d_0}$, $d_0>0$, where $K\ge 1$ is independent of $T$. Consequently,
 \begin{eqnarray}\label{T45}
&{}&\int_{Q_T}|v(t,x)|^p\varphi(t,x)\,dt\,dx\nonumber\\
&{}&\lesssim \left(\int_{\widetilde{Q}_T}|v(t,x)|^{p}\varphi(t,x)\,dt\,dx\right)^{1/qp}\left(T^{\frac{\sigma_3(pq-1)}{pq}}K^{\frac{d_0d(pq-1)}{pq}}+T^{\frac{\sigma_4(pq-1)}{pq}}K^{\frac{d_0d(pq-1)}{pq}-\mu d_0} \right)\nonumber\\
&{}&+\,T^{\sigma_1}K^{d_0d-\frac{\sigma d_0p}{pq-1}}+T^{\sigma_2}K^{d_0d-\frac{ pd_0(\sigma+\mu q)}{pq-1}}\nonumber\\
&{}&+\,\left(\int_{\mathbb{R}^d}|v_0(x)|\,dx\right)^{1/q}\left(T^{-\frac{1}{q}+\frac{d_0d(q-1)}{q}+\frac{q-1}{q}-\gamma}K^{\frac{d_0d(q-1)}{q}}+T^{-\frac{1}{q}+\frac{d_0d(q-1)}{q}+\frac{q-1}{q}-\mu d_0}K^{\frac{d_0d(q-1)}{q}-\mu d_0}\right)\nonumber\\
&{}&+\,T^{-\gamma}\int_{\mathbb{R}^d}|u_0(x)|\,dx.
\end{eqnarray}
As Lemma \ref{lemma6} and $\displaystyle d=\overline{D}$, imply that $\sigma_i=0$ for all $1\leq i\leq 4$, we infer from \eqref{T45} that
 \begin{eqnarray}\label{T46}
\int_{Q_T}|v(t,x)|^p\varphi(t,x)\,dt\,dx&\lesssim& K^{\frac{d_0d(pq-1)}{pq}}\left(\int_{\widetilde{Q}_T}|v(t,x)|^{p}\varphi(t,x)\,dt\,dx\right)^{1/qp}\nonumber\\
&{}&+\,K^{\frac{1-pq(1-\gamma)}{pq-1}}+K^{-1}\nonumber\\
&{}&+\,T^{\frac{(1+\gamma q)(1-p)}{q(pq-1)}}K^{\frac{d_0d(q-1)}{q}}\left(\int_{\mathbb{R}^d}|v_0(x)|\,dx\right)^{1/q}\nonumber\\
&{}&+\,T^{-\gamma}\int_{\mathbb{R}^d}|u_0(x)|\,dx.
\end{eqnarray}
On the other hand, using \eqref{T47},\eqref{T48} and \eqref{T18} with $T\rightarrow\infty$, taking account of $\sigma_i=0$ for all $1\leq i\leq 4$, we get
 $$
v\in L^p((0,\infty),L^p(\mathbb{R}^d));
$$
consequently,
\begin{equation}\label{T49}
\lim_{T\rightarrow\infty}\int_{\widetilde{Q}_T}|v(t,x)|^{p}\varphi(t,x)\,dt\,dx=\lim_{T\rightarrow\infty}\left(\int_{Q_T}|v(t,x)|^{p}\varphi(t,x)\,dt\,dx-\int_{Q_{T/2}}|v(t,x)|^{p}\varphi(t,x)\,dt\,dx\right)=0.
\end{equation}
So, letting $T\rightarrow+\infty$ into \eqref{T46} and using \eqref{T49}, we arrive at
$$
\int_0^\infty\int_{\mathbb{R}^d}|v(t,x)|^p\,dt\,dx \lesssim K^{\frac{1-pq(1-\gamma)}{pq-1}}+K^{-1},
$$
which leads, by letting $K\rightarrow+\infty$, to $v\equiv0$ a.e.. Combing it with \eqref{T42}, we derive that 
$$\int_{Q_T}|u(t,x)|^q\varphi(t,x)\,dt\,dx\lesssim T^{-1}\int_{\mathbb{R}^d}|v_0(x)|\,dx,$$
which yields
$$\lim_{T\rightarrow\infty}\int_{Q_T}|u(t,x)|^q\varphi(t,x)\,dt\,dx=0,$$
i.e. $u\equiv0$ a.e.; contradiction.\\

\noindent{\bf Case 4: $\displaystyle d=\overline{D}$, $\theta=1$, and $\gamma\leq1$}.\\
  In this case, we also suppose that $\gamma\leq\frac{p(q-1)}{q(p-1)}$. Let $\widetilde{\phi}$ be as in the Case $3$. Using $\theta=1$ and $\widetilde{\phi}^\ell(t)$, $\ell>p^{\prime}$, instead of $\phi(t)$, we may improve \eqref{T11} as follows
 \begin{equation}\label{T}
\int_{Q_T}v(t,x)\,\varphi_t(t,x)\,dt\,dx\leq\left(\int_{\widetilde{Q}_T}|v(t,x)|^{p}\varphi(t,x)\,dt\,dx\right)^{1/p}\left(\int_{Q_T}\varphi^{-\frac{1}{p-1}}(t,x) \left|\varphi_t(t,x)\right|^{p^\prime}\,dt\,dx\right)^{1/p^\prime},
\end{equation}
and so \eqref{T17} becomes
\begin{eqnarray*}
c_0\int_{Q_T}|u(t,x)|^q\varphi(t,x)\,dt\,dx&\leq& \left(\int_{\widetilde{Q}_T}|v(t,x)|^{p}\varphi(t,x)\,dt\,dx\right)^{1/p}\left(\int_{Q_T}\varphi^{-\frac{1}{p-1}}(t,x) \left|\varphi_t(t,x)\right|^{p^\prime}\,dt\,dx\right)^{1/p^\prime}\\
&{}&+\,\left(\int_{Q_T}|v(t,x)|^{p}\varphi(t,x)\,dt\,dx\right)^{1/p}\left(\int_{Q_T}\varphi^{-\frac{1}{p-1}}(t,x) \left|(-\Delta)^{\sigma/2}\varphi(t,x)\right|^{p^\prime}\,dt\,dx\right)^{1/p^\prime}\\
&{}&+\,C\,T^{-1}\int_{\mathbb{R}^d}|v_0(x)|\,dx,
\end{eqnarray*}
Set $R=K^{d_0}T^{d_0}$, $d_0>0$, where $K\ge 1$ is independent of $T$. Consequently,
 \begin{eqnarray}\label{T62}
\int_{Q_T}|u(t,x)|^q\varphi(t,x)\,dt\,dx&\lesssim& T^{-1+\frac{(d_0d+1)(p-1)}{p}}K^{\frac{d_0d(p-1)}{p}}\left(\int_{\widetilde{Q}_T}|v(t,x)|^{p}\varphi(t,x)\,dt\,dx\right)^{1/p}\nonumber\\
&{}&+\,T^{-\sigma d_0+\frac{(d_0d+1)(p-1)}{p}}K^{-\sigma d_0+\frac{d_0d(p-1)}{p}}\left(\int_{Q_T}|v(t,x)|^{p}\varphi(t,x)\,dt\,dx\right)^{1/p}\nonumber\\
&{}&+\,T^{-1}\int_{\mathbb{R}^d}|v_0(x)|\,dx.
\end{eqnarray}
As Lemma \ref{lemma6} and $\displaystyle d=\overline{D}$, imply that $\sigma_i=0$ for all $1\leq i\leq 4$, we infer from \eqref{T62} that
 \begin{eqnarray*}
\int_{Q_T}|u(t,x)|^q\varphi(t,x)\,dt\,dx&\lesssim& T^{\frac{\gamma q(p-1)-p(q-1)}{pq-1}}K^{\frac{d_0d(p-1)}{p}}\left(\int_{\widetilde{Q}_T}|v(t,x)|^{p}\varphi(t,x)\,dt\,dx\right)^{1/p}\\
&{}&+\,T^{\frac{\gamma q(p-1)-p(q-1)}{pq-1}}K^{\frac{\gamma q(p-1)-p(q-1)}{pq-1}-\frac{p-1}{p}}\left(\int_{Q_T}|v(t,x)|^{p}\varphi(t,x)\,dt\,dx\right)^{1/p}\\
&{}&+\,T^{-1}\int_{\mathbb{R}^d}|v_0(x)|\,dx.
\end{eqnarray*}
Then, by using the fact that $\gamma\leq\frac{p(q-1)}{q(p-1)}$, we may conclude
\begin{eqnarray}\label{T63}
\int_{Q_T}u(t,x)|^q\varphi(t,x)\,dt\,dx&\lesssim& K^{\frac{d_0d(p-1)}{p}}\left(\int_{\widetilde{Q}_T}|v(t,x)|^{p}\varphi(t,x)\,dt\,dx\right)^{1/p}\nonumber\\
&{}&+\,K^{-\frac{p-1}{p}}\left(\int_{Q_T}|v(t,x)|^{p}\varphi(t,x)\,dt\,dx\right)^{1/p}\nonumber\\
&{}&+\,T^{-1}\int_{\mathbb{R}^d}|v_0(x)|\,dx.
\end{eqnarray}
On the other hand, using \eqref{T47},\eqref{T48} and \eqref{T18} with $T\rightarrow\infty$, taking account of $\sigma_i=0$ for all $1\leq i\leq 4$, we get
 $$
v\in L^p((0,\infty),L^p(\mathbb{R}^d));
$$
consequently,
\begin{equation}\label{T64}
\lim_{T\rightarrow\infty}\int_{\widetilde{Q}_T}|v(t,x)|^{p}\varphi(t,x)\,dt\,dx=\lim_{T\rightarrow\infty}\left(\int_{Q_T}|v(t,x)|^{p}\varphi(t,x)\,dt\,dx-\int_{Q_{T/2}}|v(t,x)|^{p}\varphi(t,x)\,dt\,dx\right)=0.
\end{equation}
So, letting $T\rightarrow+\infty$ into \eqref{T63} and using \eqref{T64}, we arrive at
$$
\int_0^\infty\int_{\mathbb{R}^d}|u(t,x)|^q\,dt\,dx \lesssim K^{-\frac{p-1}{p}},
$$
which leads, by letting $K\rightarrow+\infty$, to $u\equiv0$ a.e.. Combing it with \eqref{T16}, we derive that 
$$
c_0\int_{Q_T}|v(t,x)|^p\varphi(t,x)\,dt\,dx\lesssim T^{-\gamma}\int_{\mathbb{R}^d}|u_0(x)|\,dx,
$$
which yields
$$\lim_{T\rightarrow\infty}\int_{Q_T}|v(t,x)|^p\varphi(t,x)\,dt\,dx=0,$$
i.e. $v\equiv0$ a.e.; contradiction.\\

\noindent{\bf Case 5: $\displaystyle d=\overline{E}$, $\gamma=1$ and $\theta\neq1$}.\\
  In this case, we also suppose that $pq>1/(1-\theta)$. Let $\widetilde{\phi}$ be as in the Case $3$. Using $\gamma=1$ and $\widetilde{\phi}^\ell(t)$, $\ell>p^{\prime}$, instead of $\phi(t)$, we may improve \eqref{T10} as follows
 \begin{equation}\label{T50}
\int_{Q_T}u(t,x)\,\varphi_t(t,x)\,dt\,dx\leq\left(\int_{\widetilde{Q}_T}|u(t,x)|^{q}\varphi(t,x)\,dt\,dx\right)^{1/q}\left(\int_{Q_T}\varphi^{-\frac{1}{q-1}}(t,x) \left|\varphi_t(t,x)\right|^{q^\prime}\,dt\,dx\right)^{1/q^\prime},
\end{equation}
and so \eqref{T16} becomes
\begin{eqnarray}\label{T51}
c_0\int_{Q_T}|v(t,x)|^p\varphi(t,x)\,dt\,dx&\leq& \left(\int_{\widetilde{Q}_T}|u(t,x)|^{q}\varphi(t,x)\,dt\,dx\right)^{1/q}\left(\int_{Q_T}\varphi^{-\frac{1}{q-1}}(t,x) \left|\varphi_t(t,x)\right|^{q^\prime}\,dt\,dx\right)^{1/q^\prime}\nonumber\\
&{}&+\,\left(\int_{Q_T}|u(t,x)|^{q}\varphi(t,x)\,dt\,dx\right)^{1/q}\left(\int_{Q_T}\varphi^{-\frac{1}{q-1}}(t,x) \left|(-\Delta)^{\mu/2}\varphi(t,x)\right|^{q^\prime}\,dt\,dx\right)^{1/q^\prime}\nonumber\\
&{}&+\,C\,T^{-1}\int_{\mathbb{R}^d}|u_0(x)|\,dx,
\end{eqnarray}
Now, inserting \eqref{T51} into \eqref{T17}, we arrive at
\begin{eqnarray}\label{T52}
\int_{Q_T}|u(t,x)|^q\varphi(t,x)\,dt\,dx&\lesssim&\mathcal{B}\left(\int_{\widetilde{Q}_T}|u(t,x)|^{q}\varphi(t,x)\,dt\,dx\right)^{1/pq}\left(\int_{Q_T}\varphi^{-\frac{1}{q-1}}(t,x) \left|\varphi_t(t,x)\right|^{q^\prime}\,dt\,dx\right)^{1/pq^\prime}\nonumber\\
&{}&+\,\left(\int_{Q_T}|u(t,x)|^{q}\varphi(t,x)\,dt\,dx\right)^{1/pq}\mathcal{B}\left(\int_{Q_T}\varphi^{-\frac{1}{q-1}}(t,x) \left|(-\Delta)^{\mu/2}\varphi(t,x)\right|^{q^\prime}\,dt\,dx\right)^{1/pq^\prime}\nonumber\\
&{}&+\,T^{-\frac{1}{p}}\mathcal{B}\left(\int_{\mathbb{R}^d}|u_0(x)|\,dx\right)^{1/p}+\,T^{-\theta}\int_{\mathbb{R}^d}|v_0(x)|\,dx.
\end{eqnarray}
Applying the following Young's inequality
$$
AB\leq\frac{1}{2}A^{pq}+C(p,q)B^{\frac{pq}{pq-1}},\quad A\geq0,\;B\geq0,
$$
on the second term in the right-hand side of \eqref{T52}, we arrive at
\begin{eqnarray}\label{T53}
\frac{1}{2}\int_{Q_T}|u(t,x)|^q\varphi(t,x)\,dt\,dx&\lesssim&\mathcal{B}\left(\int_{\widetilde{Q}_T}|u(t,x)|^{q}\varphi(t,x)\,dt\,dx\right)^{1/pq}\left(\int_{Q_T}\varphi^{-\frac{1}{q-1}}(t,x) \left|\varphi_t(t,x)\right|^{q^\prime}\,dt\,dx\right)^{1/pq^\prime}\nonumber\\
&{}&+\,\mathcal{B}^{\frac{pq}{pq-1}}\left(\int_{Q_T}\varphi^{-\frac{1}{q-1}}(t,x) \left|(-\Delta)^{\mu/2}\varphi(t,x)\right|^{q^\prime}\,dt\,dx\right)^{\frac{q-1}{pq-1}}\nonumber\\
&{}&+\,T^{-\frac{1}{p}}\mathcal{B}\left(\int_{\mathbb{R}^d}|u_0(x)|\,dx\right)^{1/p}+\,T^{-\theta}\int_{\mathbb{R}^d}|v_0(x)|\,dx.
\end{eqnarray}
At this stage, we set $R=K^{d_0}T^{d_0}$, $d_0>0$, where $K\ge 1$ is independent of $T$. Consequently,
 \begin{eqnarray}\label{T54}
&{}&\int_{Q_T}|u(t,x)|^q\varphi(t,x)\,dt\,dx\nonumber\\
&{}&\lesssim \left(\int_{\widetilde{Q}_T}|u(t,x)|^{q}\varphi(t,x)\,dt\,dx\right)^{1/pq}\left(T^{\frac{\rho_1(pq-1)}{pq}}K^{\frac{d_0d(pq-1)}{pq}}+T^{\frac{\rho_2(pq-1)}{pq}}K^{\frac{d_0d(pq-1)}{pq}-\sigma d_0} \right)\nonumber\\
&{}&+\,\left(T^{\rho_3}K^{d_0d-\frac{\mu d_0q}{pq-1}}+T^{\rho_4}K^{d_0d-\frac{qd_0(\mu+\sigma p)}{pq-1}}\right)\nonumber\\
&{}&+\,\left(\int_{\mathbb{R}^d}|u_0(x)|\,dx\right)^{1/p}\left(T^{-\frac{1}{p}+\frac{d_0d(p-1)}{p}+\frac{p-1}{p}-\theta}K^{\frac{d_0d(p-1)}{p}}+T^{-\frac{1}{p}+\frac{d_0d(p-1)}{p}+\frac{p-1}{p}-\sigma d_0}K^{\frac{d_0d(p-1)}{p}-\sigma d_0}\right)\nonumber\\
&{}&+\,T^{-\theta}\int_{\mathbb{R}^d}|v_0(x)|\,dx.
\end{eqnarray}
As Lemma \ref{lemma7} and $\displaystyle d=\overline{E}$, imply that $\rho_i=0$ for all $1\leq i\leq 4$, we infer from \eqref{T54} that
 \begin{eqnarray}\label{T57}
\int_{Q_T}|u(t,x)|^q\varphi(t,x)\,dt\,dx&\lesssim& K^{\frac{d_0d(pq-1)}{pq}}\left(\int_{\widetilde{Q}_T}|u(t,x)|^{q}\varphi(t,x)\,dt\,dx\right)^{1/pq}\nonumber\\
&{}&+\,K^{\frac{1-pq(1-\theta)}{pq-1}}+K^{-1}\nonumber\\
&{}&+\,T^{\frac{(1+\theta p)(1-q)}{p(pq-1)}}K^{\frac{d_0d(p-1)}{p}}\left(\int_{\mathbb{R}^d}|u_0(x)|\,dx\right)^{1/p}\nonumber\\
&{}&+\,T^{-\theta}\int_{\mathbb{R}^d}|v_0(x)|\,dx.
\end{eqnarray}
On the other hand, using \eqref{T55},\eqref{T56} and \eqref{T19} with $T\rightarrow\infty$, taking account of $\rho_i=0$ for all $1\leq i\leq 4$, we get
 $$
u\in L^q((0,\infty),L^q(\mathbb{R}^d));
$$
consequently,
\begin{equation}\label{T58}
\lim_{T\rightarrow\infty}\int_{\widetilde{Q}_T}|u(t,x)|^{q}\varphi(t,x)\,dt\,dx=\lim_{T\rightarrow\infty}\left(\int_{Q_T}|u(t,x)|^{q}\varphi(t,x)\,dt\,dx-\int_{Q_{T/2}}|u(t,x)|^{q}\varphi(t,x)\,dt\,dx\right)=0.
\end{equation}
So, letting $T\rightarrow+\infty$ into \eqref{T57} and using \eqref{T58}, we arrive at
$$
\int_0^\infty\int_{\mathbb{R}^d}|u(t,x)|^q\,dt\,dx \lesssim K^{\frac{1-pq(1-\theta)}{pq-1}}+K^{-1},
$$
which leads, by letting $K\rightarrow+\infty$, to $u\equiv0$ a.e.. Combing it with \eqref{T51}, we derive that 
$$\int_{Q_T}|v(t,x)|^p\varphi(t,x)\,dt\,dx\lesssim T^{-1}\int_{\mathbb{R}^d}|u_0(x)|\,dx,$$
which yields
$$\lim_{T\rightarrow\infty}\int_{Q_T}|v(t,x)|^p\varphi(t,x)\,dt\,dx=0,$$
i.e. $v\equiv0$ a.e.; contradiction.\\

\noindent{\bf Case 6: $\displaystyle d=\overline{E}$, $\gamma=1$, and $\theta\leq1$}.\\
  In this case, we also suppose that $\theta\leq\frac{q(p-1)}{p(q-1)}$. Let $\widetilde{\phi}$ be as in the Case $3$. Using $\gamma=1$ and $\widetilde{\phi}^\ell(t)$, $\ell>p^{\prime}$, instead of $\phi(t)$, we may improve \eqref{T10} as follows
 \begin{equation}\label{T}
\int_{Q_T}u(t,x)\,\varphi_t(t,x)\,dt\,dx\leq\left(\int_{\widetilde{Q}_T}|u(t,x)|^{q}\varphi(t,x)\,dt\,dx\right)^{1/q}\left(\int_{Q_T}\varphi^{-\frac{1}{q-1}}(t,x) \left|\varphi_t(t,x)\right|^{q^\prime}\,dt\,dx\right)^{1/q^\prime},
\end{equation}
and so \eqref{T16} becomes
\begin{eqnarray*}
c_0\int_{Q_T}|v(t,x)|^p\varphi(t,x)\,dt\,dx&\leq& \left(\int_{\widetilde{Q}_T}|u(t,x)|^{q}\varphi(t,x)\,dt\,dx\right)^{1/q}\left(\int_{Q_T}\varphi^{-\frac{1}{q-1}}(t,x) \left|\varphi_t(t,x)\right|^{q^\prime}\,dt\,dx\right)^{1/q^\prime}\\
&{}&+\,\left(\int_{Q_T}|u(t,x)|^{q}\varphi(t,x)\,dt\,dx\right)^{1/q}\left(\int_{Q_T}\varphi^{-\frac{1}{q-1}}(t,x) \left|(-\Delta)^{\mu/2}\varphi(t,x)\right|^{q^\prime}\,dt\,dx\right)^{1/q^\prime}\\
&{}&+\,C\,T^{-1}\int_{\mathbb{R}^d}|u_0(x)|\,dx,
\end{eqnarray*}
Set $R=K^{d_0}T^{d_0}$, $d_0>0$, where $K\ge 1$ is independent of $T$. Consequently,
 \begin{eqnarray}\label{T59}
\int_{Q_T}|v(t,x)|^p\varphi(t,x)\,dt\,dx&\lesssim&T^{-1+\frac{(d_0d+1)(q-1)}{q}}K^{\frac{d_0d(q-1)}{q}}\left(\int_{\widetilde{Q}_T}|u(t,x)|^{q}\varphi(t,x)\,dt\,dx\right)^{1/q}\nonumber\\
&{}&+\,T^{-\mu d_0+\frac{(d_0d+1)(q-1)}{q}}K^{-\mu d_0+\frac{d_0d(q-1)}{q}}\left(\int_{Q_T}|u(t,x)|^{q}\varphi(t,x)\,dt\,dx\right)^{1/q}\nonumber\\
&{}&+\,T^{-1}\int_{\mathbb{R}^d}|u_0(x)|\,dx.
\end{eqnarray}
As Lemma \ref{lemma7} and $\displaystyle d=\overline{E}$, imply that $\rho_i=0$ for all $1\leq i\leq 4$, we infer from \eqref{T59} that
 \begin{eqnarray*}
\int_{Q_T}|v(t,x)|^p\varphi(t,x)\,dt\,dx&\lesssim& T^{\frac{\theta p(q-1)-q(p-1)}{pq-1}}K^{\frac{d_0d(q-1)}{q}}\left(\int_{\widetilde{Q}_T}|u(t,x)|^{q}\varphi(t,x)\,dt\,dx\right)^{1/q}\\
&{}&+\,T^{\frac{\theta p(q-1)-q(p-1)}{pq-1}}K^{\frac{\theta p(q-1)-q(p-1)}{pq-1}-\frac{q-1}{q}}\left(\int_{Q_T}|u(t,x)|^{q}\varphi(t,x)\,dt\,dx\right)^{1/q}\\
&{}&+\,T^{-1}\int_{\mathbb{R}^d}|u_0(x)|\,dx.
\end{eqnarray*}
Then, by using the fact that $\theta\leq\frac{q(p-1)}{p(q-1)}$, we may conclude
\begin{eqnarray}\label{T60}
\int_{Q_T}|v(t,x)|^p\varphi(t,x)\,dt\,dx&\lesssim& K^{\frac{d_0d(q-1)}{q}}\left(\int_{\widetilde{Q}_T}|u(t,x)|^{q}\varphi(t,x)\,dt\,dx\right)^{1/q}\nonumber\\
&{}&+\,K^{-\frac{q-1}{q}}\left(\int_{Q_T}|u(t,x)|^{q}\varphi(t,x)\,dt\,dx\right)^{1/q}\nonumber\\
&{}&+\,T^{-1}\int_{\mathbb{R}^d}|u_0(x)|\,dx.
\end{eqnarray}
On the other hand, using \eqref{T55},\eqref{T56} and \eqref{T19} with $T\rightarrow\infty$, taking account of $\rho_i=0$ for all $1\leq i\leq 4$, we get
 $$
u\in L^q((0,\infty),L^q(\mathbb{R}^d));
$$
consequently,
\begin{equation}\label{T61}
\lim_{T\rightarrow\infty}\int_{\widetilde{Q}_T}|u(t,x)|^{q}\varphi(t,x)\,dt\,dx=\lim_{T\rightarrow\infty}\left(\int_{Q_T}|u(t,x)|^{q}\varphi(t,x)\,dt\,dx-\int_{Q_{T/2}}|u(t,x)|^{q}\varphi(t,x)\,dt\,dx\right)=0.
\end{equation}
So, letting $T\rightarrow+\infty$ into \eqref{T60} and using \eqref{T61}, we arrive at
$$
\int_0^\infty\int_{\mathbb{R}^d}|v(t,x)|^p\,dt\,dx \lesssim K^{-\frac{q-1}{q}},
$$
which leads, by letting $K\rightarrow+\infty$, to $v\equiv0$ a.e.. Combing it with \eqref{T17}, we derive that 
$$
c_0\int_{Q_T}|u(t,x)|^q\varphi(t,x)\,dt\,dx \lesssim T^{-\theta}\int_{\mathbb{R}^d}|v_0(x)|\,dx,
$$
which yields
$$\lim_{T\rightarrow\infty}\int_{Q_T}|u(t,x)|^q\varphi(t,x)\,dt\,dx=0,$$
i.e. $u\equiv0$ a.e.; contradiction.\\
 
${}$\hfill$\blacksquare$\\


\section{Proof of Theorem \ref{theo3} }\label{sec5}

\noindent {\bf Proof of Theorem \ref{theo3}}. Let $u$ be a global nontrivial weak solution of \eqref{14}. Then
\begin{eqnarray}\label{weaksolution4}
&{}&\int_{Q_T}|u(t,g(x))|^p\varphi(t,x)\,dt\,dx+\int_{Q_T}u_0(x)\,\left[D^\alpha_{t|T}\varphi(t,x)+D^{\beta}_{t|T}\varphi(t,x)\right]\,dt\,dx+\int_{Q_T}u_1(x)\,D^{\beta-1}_{t|T}\varphi(t,x)\,dt\,dx\nonumber\\
&{}&=\int_{Q_T}u(t,x)\,D^{\beta}_{t|T}\varphi(t,x)\,dt\,dx+\int_{Q_T}u(t,x)\,D^\alpha_{t|T}\varphi(t,x)\,dt\,dx+\int_{Q_T}u(t,x)(-\Delta)^{\delta/2}\varphi(t,x)\,dt\,dx,
\end{eqnarray}
holds for all $\varphi\in Y_{\delta,T}$ and all $T>0$. By introducing $\varphi^{1/p}\varphi^{-1/p}$ and applying the following Young's inequality
\[
AB\leq\frac{c_0}{6}A^p+C(p,c_0)B^{p^\prime},\quad A\geq0,\;B\geq0,\;p+p^\prime=p p^\prime,
\]
where $c_0$ is introduced in $(\textup{A}1)$, we get
\begin{equation}\label{T5}
\int_{Q_T}u(t,x)\,D^\beta_{t|T}\varphi(t,x)\,dt\,dx\leq\frac{c_0}{6}\int_{Q_T}|u(t,x)|^{p}\varphi(t,x)\,dt\,dx+C\int_{Q_T}\varphi^{-\frac{1}{p-1}}(t,x) \left|D^{\beta}_{t|T}\varphi(t,x)\right|^{p^\prime}\,dt\,dx,
\end{equation}
\begin{equation}\label{T6}
\int_{Q_T}u(t,x)\,D^\alpha_{t|T}\varphi(t,x)\,dt\,dx\leq\frac{c_0}{6}\int_{Q_T}|u(t,x)|^{p}\varphi(t,x)\,dt\,dx+C\int_{Q_T}\varphi^{-\frac{1}{p-1}}(t,x) \left|D^{\alpha}_{t|T}\varphi(t,x)\right|^{p^\prime}\,dt\,dx,
\end{equation}
and
\begin{equation}\label{T7}
\int_{Q_T}u(t,x)(-\Delta)^{\delta/2}\varphi(t,x)\,dt\,dx\leq\frac{c_0}{6}\int_{Q_T}|u(t,x)|^{p}\varphi(t,x)\,dt\,dx+C\int_{Q_T}\varphi^{-\frac{1}{p-1}}(t,x) \left|(-\Delta)^{\delta/2}\varphi(t,x)\right|^{p^\prime}\,dt\,dx.
\end{equation}

Let
$$\varphi(x,t)=\Phi_R(x)\phi(t),$$
where $\Phi_R$ is defined in Lemma \ref{lemma3} for $R>0$, and $\phi$ is defined in \eqref{w}. Observe that, using 
$(\textup{A}1)$-$(\textup{A}2)$ and the monotonicity of $\Phi_R$, we obtain the estimate
\begin{eqnarray}\label{T8}
\int_{Q_T}|u(t,g(x))|^p\varphi(t,x)\,dt\,dx&=&\int_{Q_T}|u(t,x)|^p\phi(t)\Phi_R(g^{-1}(x))|J_g^{-1}(x)|\,dt\,dx\nonumber\\
&\geq& c_0\int_{Q_T}|u(t,x)|^p\phi(t)\Phi_R(x)\,dt\,dx
\end{eqnarray}
Using the estimates \eqref{T5}-\eqref{T7} into \eqref{weaksolution4} we get
\begin{eqnarray*}
&{}&\frac{c_0}{2}\int_{Q_T}|u(t,x)|^p\varphi(t,x)\,dt\,dx+\int_{Q_T}u_0(x)\,\left[D^\alpha_{t|T}\varphi(t,x)+D^{\beta}_{t|T}\varphi(t,x)\right]\,dt\,dx+\int_{Q_T}u_1(x)\,D^{\beta-1}_{t|T}\varphi(t,x)\,dt\,dx\nonumber\\
&{}&\lesssim \int_{Q_T}\varphi^{-\frac{1}{p-1}}(t,x) \left( \left|D^{\beta}_{t|T}\phi(t,x)\right|^{p^\prime}+\left|D^{\alpha}_{t|T}\phi(t,x)\right|^{p^\prime}\right)\,dt\,dx+\,\int_{Q_T}\varphi^{-\frac{1}{p-1}}(t,x) \left|(-\Delta)^{\delta/2}\varphi(t,x)\right|^{p^\prime}\,dt\,dx.
\end{eqnarray*}
Whereupon, using Lemmas \ref{lemma2} and \ref{lemma5}, we arrive at
\begin{eqnarray}\label{17}
\int_{Q_T}|u(t,x)|^p\varphi(t,x)\,dt\,dx&\lesssim& \int_{\mathbb{R}^d}|u_0(x)|\,dx\int_0^T \left[D^\alpha_{t|T}\phi(t)+D^{\beta}_{t|T}\phi(t)\right]\,dt+ \int_{\mathbb{R}^d}|u_1(x)|\,dx\int_0^T D^{\beta-1}_{t|T}\phi(t)\,dt\nonumber\\
&{}&+\,\int_{\mathbb{R}^d}\Phi_R(x)\,dx\int_0^T\phi^{-\frac{1}{p-1}}(t)\left( \left|D^{\beta}_{t|T}\phi(t,x)\right|^{p^\prime}+\left|D^{\alpha}_{t|T}\phi(t,x)\right|^{p^\prime}\right)\,dt\nonumber\\
&{}&+\,\int_0^T\phi(t)\,dt\int_{\mathbb{R}^d}\Phi_R^{-\frac{1}{p-1}}(x) \left|(-\Delta)^{\delta/2}\Phi_R(x)\right|^{p^\prime}\,dx\nonumber\\
&\lesssim& T^{-\alpha} \int_{\mathbb{R}^d}|u_0(x)|\,dx+ C\,T^{-(\beta-1)} \int_{\mathbb{R}^d}|u_1(x)|\,dx\nonumber\\
&{}&+\,R^d\,T^{1-\beta\frac{p}{p-1}}+\,R^d\,T^{1-\alpha\frac{p}{p-1}} +\,T\,R^{-\frac{\delta p}{p-1}+d}.
\end{eqnarray}
At this stage, two cases can be distinguished.\\
\noindent {\bf Case 1}: If $p<p_*$, we set $R:=T^{\alpha/\delta}$, then \eqref{17} implies
\begin{equation}\label{18}
\int_{Q_T}|u(t,x)|^p\varphi(t,x)\,dt\,dx\lesssim T^{-\alpha} \int_{\mathbb{R}^d}|u_0(x)|\,dx+\,T^{-(\beta-1)} \int_{\mathbb{R}^d}|u_1(x)|\,dx+\,T^{1-\alpha\frac{p}{p-1}+\frac{\alpha d}{\delta}}.
\end{equation}
Letting $T\rightarrow+\infty$, using the fact that $p<p_*\Longleftrightarrow 1-\alpha\frac{p}{p-1}+\frac{\alpha d}{\delta}<0$, the assumption $u_0 \in L^1$ and Lebesgue's dominated convergence theorem, we conclude that
\begin{equation}
\int_0^\infty\int_{\mathbb{R}^d}|u(t,x)|^p\,dt\,dx\leq 0,
\end{equation}
which leads to a contradiction.\\
\noindent {\bf Case 2}: If $p=p_*$ and $\alpha=1$. Let $\widetilde{\phi}$ be a smooth nonnegative non-increasing function such that  $0\leq \widetilde{\phi} \leq 1$ and
\[
\widetilde{\phi}(t)=\left\{\begin {array}{ll}\displaystyle{1}&\displaystyle{\quad\text{if }0\leq t\leq 1/2,}\\
{}\\
\displaystyle{0}&\displaystyle{\quad\text {if }t\geq 1.}
\end {array}\right.
\]
Using $\widetilde{\phi}^\ell(t)$, $\ell>p^{\prime}$, instead of $\phi(t)$ and applying H\"{o}lder's inequality instead of Young's inequality into \eqref{T6}, we obtain
\begin{eqnarray}\label{T9}
\int_{Q_T}u(t,x)\,\varphi_t(t,x)\,dt\,dx&\leq& \left(\int_{\widetilde{Q}_T}|u(t,x)|^{p}\varphi(t,x)\,dt\,dx\right)^{1/p}\left(\int_{Q_T}\varphi^{-\frac{1}{p-1}}(t,x) \left|\varphi_t(t,x)\right|^{p^\prime}\,dt\,dx\right)^{1/p^\prime}\nonumber\\
&=&\left(\int_{\widetilde{Q}_T}|u(t,x)|^{p}\varphi(t,x)\,dt\,dx\right)^{1/p}\left(\int_{\mathbb{R}^d}\Phi_R(x)\,dx\int_0^T\widetilde{\phi}^{\ell -p^\prime}(t)\left|\frac{d}{dt}\widetilde{\phi}(t)\right|^{p^\prime}\,dt\right)^{1/p^\prime}\nonumber\\
&\lesssim& T^{-1+\frac{1}{p^\prime}}\,R^{\frac{d}{p^\prime}}\left(\int_{\widetilde{Q}_T}|u(t,x)|^{p}\varphi(t,x)\,dt\,dx\right)^{1/p}
\end{eqnarray}
where $\widetilde{Q}_T=[T/2,T]\times \mathbb{R}^d$. We set $R:=K^{\alpha/\delta}T^{\alpha/\delta}$,  where $K\ge 1$ is independent of $T$. Using the estimates \eqref{T5}, \eqref{T7} and \eqref{T9} into \eqref{weaksolution4} and taking into account the fact that $p=p_*$, we get
\begin{eqnarray}\label{19}
\int_{Q_T}|u(t,x)|^p\varphi(t,x)\,dt\,dx &\lesssim& T^{-\alpha} \int_{\mathbb{R}^d}|u_0(x)|\,dx+\,T^{-(\beta-1)} \int_{\mathbb{R}^d}|u_1(x)|\,dx+\,R^d\,T^{1-\beta\frac{p}{p-1}} \nonumber\\
&{}&\,+\,T\,R^{-\frac{\delta p}
{p-1}+d}+\,T^{-1+\frac{1}{p^\prime}}\,R^{\frac{d}{p^\prime}}\left(\int_{\widetilde{Q}_T}|u(t,x)|^{p}\varphi(t,x)\,dt\,dx\right)^{1/p} \nonumber\\
&=&  T^{-\alpha} \int_{\mathbb{R}^d}|u_0(x)|\,dx+\,T^{-(\beta-1)} \int_{\mathbb{R}^d}|u_1(x)|\,dx+\,K^{\alpha d/\delta}T^{-(\beta-\alpha)\frac{p}{p-1}} \nonumber\\
&{}& \,+\,K^{-1}+C\,K^{\frac{d}{p^\prime}}\left(\int_{\widetilde{Q}_T}|u(t,x)|^{p}\varphi(t,x)\,dt\,dx\right)^{1/p}.
\end{eqnarray}
On the other hand, using \eqref{18} with $T\rightarrow\infty $ and taking account of $p=p_*,$ we obtain
 $$
u\in L^p((0,\infty),L^p(\mathbb{R}^d));
$$
whereupon,
\begin{equation}\label{20}
\lim_{T\rightarrow\infty}\int_{\widetilde{Q}_T}|u(t,x)|^{p}\varphi(t,x)\,dt\,dx=\lim_{T\rightarrow\infty}\left(\int_{Q_T}|u(t,x)|^{p}\varphi(t,x)\,dt\,dx-\int_{Q_{T/2}}|u(t,x)|^{p}\varphi(t,x)\,dt\,dx\right)=0.
\end{equation}
Finally, letting $T\rightarrow+\infty$ into \eqref{19} and using \eqref{20} and the fact that $u_0,u_1\in L^1(\mathbb{R}^d)$, we arrive at
$$
\int_0^\infty\int_{\mathbb{R}^d}|u(t,x)|^p\,dt\,dx \lesssim K^{-1},
$$
which leads to a contradiction for $K\gg1$.\\
${}$\hfill$\blacksquare$\\


\bibliographystyle{elsarticle-num}



\end{document}